\documentclass{amsart}
\usepackage{amsthm}
\usepackage{amsmath, amssymb, graphicx}
\usepackage{color}
\usepackage{pst-all}

\psset{unit=1mm}

\theoremstyle{plain}
\newtheorem{prop}{Proposition}[section]
\newtheorem{conj}[prop]{Conjecture}
\newtheorem{coro}[prop]{Corollary}
\newtheorem{lemm}[prop]{Lemma}

\newtheorem{thrm}[prop]{Theorem}

\theoremstyle{definition}
\newtheorem{defi}[prop]{Definition}

\newtheorem{rema}[prop]{Remark}

\numberwithin{equation}{section}

\newdimen\arrayitem
\def\arrayitem{3.5mm}
\newcommand\SM[1]{\hbox to \arrayitem{\hfil$#1$\hfil}}
\def\0{\hbox to \arrayitem{\hfil$\cdot$\hfil}}
\def\1{\hbox to \arrayitem{\hfil$1$\hfil}}
\def\mo{\hbox to \arrayitem{\hfil$-\!1$\hfil}}
\def\2{\hbox to \arrayitem{\hfil$2$\hfil}}
\def\3{\hbox to \arrayitem{\hfil$3$\hfil}}
\def\4{\hbox to \arrayitem{\hfil$4$\hfil}}
\def\5{\hbox to \arrayitem{\hfil$5$\hfil}}
\def\6{\hbox to \arrayitem{\hfil$6$\hfil}}
\def\7{\hbox to \arrayitem{\hfil$7$\hfil}}
\def\8{\hbox to \arrayitem{\hfil$8$\hfil}}

\renewcommand\aa{a} 
\renewcommand\AA{A} 
\newcommand\Ack{\mathrm{Ack}}
\newcommand\act{\mathbin{\scriptscriptstyle\bullet}}\newcommand\bb{b} 
\newcommand\Aut{\mathrm{Aut}}
\newcommand\AxI{\mathrm{I}}
\newcommand\BB{B} 
\newcommand\BKL[1]{B_{#1}^{\hspace{-0.05em}\raise0.6pt\hbox{$\scriptscriptstyle+$}\hspace{-0.05em}*}}
\newcommand\BP[1]{B_{#1}^{\scriptscriptstyle+}}
\newcommand\br{\beta}
\newcommand\BR[1]{B_{#1}}
\newcommand\brr{\gamma}

\newcommand\cc{c}
\newcommand\CC{C}
\newcommand\CCCC{\mathbb{C}}
\newcommand\dd{d}
\newcommand\DD{D}
\newcommand\der{\partial}
\newcommand\ee{e}

\newcommand\EEE{\mathcal{E}}
\newcommand\eps{\varepsilon}
\newcommand\ff{f}
\newcommand\FF{F}
\newcommand\flip{\phi}
\newcommand\Flip[1]{\flip_#1}
\newcommand\gD{>_{\!\scriptscriptstyle D}}

\let\ge=\geqslant

\newcommand\GG{G}
\newcommand\GGG{\mathcal{G}}
\newcommand\GL{\mathrm{GL}}

\newcommand\HH{H}
\newcommand\Hom{\mathrm{Hom}}
\newcommand\id{\mathrm{id}}
\newcommand\ii{i}
\newcommand\inv{^{-1}}

\newcommand\IS[1]{\mathsf{I}\Sigma_{#1}}
\newcounter{ITEM}
\newcommand\ITEM[1]{\setcounter{ITEM}{#1}\leavevmode\hbox{\rm(\roman{ITEM})}}
\newcommand\jj{j}
\newcommand\kk{k}
\newcommand\lD{<_{\!\scriptscriptstyle D}}
\newcommand\lDf{<_{\!\scriptscriptstyle D}^{\scriptscriptstyle\flip}}
\let\le=\leqslant
\newcommand\leD{\le_{\!\scriptscriptstyle D}}

\newcommand\LL{L}
\newcommand\mm{m}
\newcommand\MM{M}
\newcommand\MOD[2]{#1\,\mathrm{mod}\,#2}
\newcommand\nn{n}
\newcommand\NN{N}
\newcommand\NNNN{\mathbb{N}}
\newcommand\op{\mathbin{*}}
\newcommand\opbar{\mathbin{\bar*}}
\newcommand\opl{\hspace{0.2ex}\rceil\hspace{0.1ex}}
\newcommand\opr{\hspace{0.1ex}\lceil\hspace{0.2ex}}
\newcommand\opt{\mathbin{\triangleleft}}
\newcommand\optbar{\mathbin{\bar\triangleleft}}
\newcommand\opp{\relax}
\newcommand\PA{\mathsf{P\hspace{-0.3ex}A}}
\newcommand\pdots{\hspace{0.2ex}{\cdot}{\cdot}{\cdot}\hspace{0.2ex}}
\newcommand\per{\pi}
\newcommand\phih{\widehat{\VR(2,0)\smash\phi}}
\newcommand\pp{p}

\newcommand\proj{\mathrm{pr}}
\newcommand\qq{q}
\newcommand\QQ{Q}
\newcommand\Rack{{\scriptscriptstyle\mathrm{R}}}
\newcommand\resp{{\it resp.}~}
\newcommand\res{\mathord{\upharpoonright}}
\newcommand\rr{r}
\newcommand\RR{R}
\newcommand\RRRR{\mathbb{R}}
\newcommand\sig[1]{\sigma_{\!#1}^{\vrule height5pt width0pt}}
\newcommand\sigg[2]{\sigma_{\!#1}^{#2}}
\newcommand\siginv[1]{\sigma_{\!#1}^{-1}}
\newcommand\sh{\mathrm{sh}}
\newcommand\SL{\mathtt{ShortLex}}

\newcommand\sol{\rho}
\renewcommand\SS{S}

\renewcommand\tt{t}
\newcommand\TT{T}

\def\VR(#1,#2){\vrule width0pt height#1mm depth#2mm}

\newcommand\VV{V}
\newcommand\wdots{, ...\hspace{0.2ex},}
\newcommand\WO[1]{W\!O_{\!#1}}
\newcommand\ww{w}
\newcommand\xx{x}
\newcommand\XX{X}
\newcommand\yy{y}

\newcommand\zz{z}
\newcommand\ZZ{Z} 
\newcommand\ZZZZ{\mathbb{Z}}

\begin{document}

\title{Laver's results and low-dimensional topology}

\author{Patrick DEHORNOY}
\address{Laboratoire de Math\'ematiques Nicolas Oresme, UMR 6139 CNRS, Universit\'e de Caen BP 5186, 14032 Caen Cedex, France}
\curraddr{Laboratoire Preuves, Programmes, Syst\`emes, UMR 7126 CNRS, Universit\'e Paris-Diderot Case 7014, 75205 Paris Cedex 13, France}
\email{dehornoy@math.unicaen.fr}
\urladdr{//www.math.unicaen.fr/\!\hbox{$\sim$}dehornoy}

\keywords{Laver table, selfdistributivity, large cardinal, rack, quandle, link invariant, rack cohomology, $R$-matrix, set-theoretic Yang--Baxter equation, birack, braid ordering, well-order, unprovability statements, conjugacy problem}

\subjclass[2010]{57M25, 03E55, 03F30, 06F15, 20F10, 20F36}

\maketitle

\begin{flushright}
{\it Dedicated to the memory of Rich Laver}
\end{flushright}

\begin{abstract}
In connection with his interest in selfdistributive algebra, Richard Laver established two deep results with potential applications in low-dimen\-sional topology, namely the existence of what is now known as the Laver tables and the well-foundedness of the standard ordering of positive braids. Here we present these results and discuss the way they could be used in topological applications. 
\end{abstract}

Richard Laver established two remarkable results that might lead to significant applications in low-dimensional topology, namely the existence of a series of finite structures satisfying the left-selfdistributive law, now known as the \emph{Laver tables}, and the well-foundedness of the standard ordering of Artin's positive braids. In this text, we shall explain the precise meaning of these results and discuss their (past or future) applications in topology. In one word, the current situation is that, although the depth of Laver's results is not questionable, few topological applications have been found. However, the example of braid groups orderability shows that, once initial obstructions are solved, topological applications of algebraic results involving selfdistributivity can be found; the situation with Laver tables is presumably similar, and the only reason explaining why so few applications are known is that no serious attempt has been made so far, mainly because the results themselves remain widely unknown in the topology community. 

Therefore this text is more a program than a report on existing results. Our aim is to provide a self-contained and accessible introduction to the subject, hopefully helping the algebraic and topological communities to better communicate. Most of the results mentioned below have already appeared in literature, a number of them even belonging to the folklore of their domain (whereas ignored outside of it). However, at least the observations about cocycles for Laver tables mentioned in Subsection~\ref{SS:Homology} (and established in another paper) are new.

Very naturally, the text comprises two sections, one devoted to Laver tables, and one devoted to the well-foundedness of the braid ordering. It should be noted that the above two topics (Laver tables, well-foundedness of the braid ordering) do not exhaust Laver's contributions to selfdistributive algebra and, from there, to potential topological applications: in particular, Laver constructed powerful tools for investigating free LD-structures, leading to applications of their own~\cite{Lvb, Lvc, Lvd}. However, connections with topology are less obvious in these cases and we shall not develop them here (see the other articles in this volume). 

\section*{Acknowledgments}
The author wishes to thank the specialists of knot theory and quantum groups who open-mindedly welcomed his questions and made a number of suggestions, in particular Nicol\'as Andruskiewitsch, Scott Carter, Mohamed Elhamdadi, Christian Kassel, Victoria Lebed, J\'ozef Przytycki, Marc Rosso.

\section{Laver tables}
\label{S:LaverTables}

The left-selfdistributivity law is the algebraic law (LD) $\xx\opp(\yy\opp\zz) = (\xx\opp\yy)\opp(\xx\opp\zz)$, which is obeyed among others by the conjugacy operation of any group. Its connection with low-dimensional topology as an algebraic distillation of Reidemeister move of type~III has been recognized more than thirty years ago \cite{Joy, Mat}, and its investigation led in particular to the discovery of the orderability of Artin's braid groups~\cite{Dfa, Dfb}. Every new example of a structure obeying the LD-law is potentially promising for topological applications. First described in~1995, the Laver tables are a family of such finite structures. Easily accessible to computer experiments but quite different from the standard examples, they are fundamental in several respects and using them in topology is one of the most exciting programs one could reasonably propose.

The section comprises four subsections: after introducing the Laver tables in Subsection~\ref{SS:Tables}, we successively discuss four approaches known to provide applications of selfidistributivity, namely diagram colourings (Subsection~\ref{SS:Colouring}), homology and cohomology (Subsection~\ref{SS:Homology}) and, finally, $R$-matrices and the Yang--Baxter equation (Subsection~\ref{SS:YBE}); in each case, we first introduce the general context and then discuss the specific case of Laver tables. To save some space, we deliberately omitted virtual knots~\cite{KauV} here, although the latter provide a natural framework for extending many existing results and might therefore appear as natural candidates for using Laver tables.

\subsection{Laver's result}
\label{SS:Tables}

In the rest of this text, an algebraic structure $(\SS, \op)$ made of a set equipped with a binary operation that obeys the LD-law will be called an \emph{LD-system}, a neutral and easily understandable term inspired by Bruck's classical textbook~\cite{Bru}. The names ``LD-groupoid'' and ``LD-algebra'' have also been used in the algebra community (conflicting with other standard meanings for ``groupoid'' and ``algebra''), whereas ``shelf'' was sporadically used in the topology community for the right-counterpart of an LD-system, that is, a binary system that obeys the right-selfdistributivity law (RD) $(\xx\opp\yy)\opp\zz = (\xx\opp\zz)\opp(\yy\opp\zz)$.

Most of the classically known LD-systems are connected with conjugacy in a group. In particular, not much was known before the 1990's about finite monogenerated LD-systems, that is, those that are generated by a single element. Richard Laver changed this situation radically in~1995---thus answering by anticipation the question candidly raised twenty years after in~\cite[Problem~9]{PrzSik}.

\begin{thrm}[Laver \cite{Lvd}]
\label{T:Tables}
\ITEM1 For every $\NN \ge 1$, there exists a unique binary operation~$\op$ on $\{1 \wdots \NN\}$ that, for all~$\pp, \qq$, satisfies 
\begin{gather}
\label{E:Tables1}
\pp \op 1 = \pp + 1 \MOD{}\NN, \\
\label{E:Tables1}
\pp \op (\qq \op 1) = (\pp \op \qq) \op (\pp \op 1).
\end{gather}
Then $(\{1 \wdots \NN\}, \op)$ is an LD-system if and only if $\NN$ is a power of~$2$.

\ITEM2 Let $\AA_\nn$ be the LD-system of size $2^\nn$ so obtained. Then, for all~$\nn$ and~$\pp \le 2^\nn$, there exists a (unique) integer~$2^\rr$ satisfying 
\begin{equation*}
\pp \op 1 < \pp \op 2 < \pdots < \pp \op 2^\rr = 2^\nn,
\end{equation*}
and the subsequent values $\pp \op \qq$ then repeat periodically. The number~$2^\rr$ is called the \emph{period} of~$\pp$, written~$\per_\nn(\pp)$; one has $\per_\nn(2^\nn - 1) = 1$ and $\per_\nn(2^\nn) = 2^\nn$.

\ITEM3  The LD-system $\AA_\nn$ admits the presentation $\langle 1 \mid 1_{[2^\nn]} = 1\rangle$, where $\xx_{[\kk]}$ stands for $(...((\xx {\op} \xx) {\op} \xx) ...) {\op} \xx$ with~$\xx$ repeated $\kk$~times.

\ITEM4 For $\nn \ge 1$, the map $\proj_\nn : \xx \mapsto \xx \ \mathrm{mod}\ 2^{\nn-1}$ defines a surjective homomorphism from~$\AA_\nn$ to~$\AA_{\nn-1}$ and, for every~$\pp \le 2^\nn$, one has either $\per_\nn(\pp) = \per_{\nn-1}(\proj_\nn(\pp))$ or $\per_\nn(\pp) = 2\per_{\nn-1}(\proj_\nn(\pp))$.

\ITEM5 If Axiom~$\AxI3$---see below---is true, the period $\per_\nn(1)$ tends to~$\infty$ with~$\nn$, the relation $\per_\nn(1) \ge \per_\nn(2)$ holds for every~$\nn$, and, letting $\AA_\infty$ be the limit of the inverse system $(\AA_\nn, \proj_\nn)_\nn$, the sub-LD-system of~$\AA_\infty$ generated by $(1, 1, ...)$ is free.
\end{thrm}

The LD-system~Ê$\AA_\nn$ is now known as the \emph{$\nn$th Laver table}. Due to their explicit definition, it is easy to effectively compute the first Laver tables, see Table~\ref{T:LaverTables}.

\begin{table}[htb]
\begin{gather*}
\begin{tabular}{c|c}
$\AA_0$&$1$\\
\hline
$1$&$1$
\end{tabular}
\quad 
\begin{tabular}{c|cc}
$\AA_1$&$1$&$2$\\
\hline
$1$&$2$&$2$\\
$2$&$1$&$2$\\
\end{tabular}
\quad 
\begin{tabular}{c|cccc}
$\AA_2$&$1$&$2$&$3$&$4$\\
\hline
$1$&$2$&$4$&$2$&$4$\\
$2$&$3$&$4$&$3$&$4$\\
$3$&$4$&$4$&$4$&$4$\\
$4$&$1$&$2$&$3$&$4$\\
\end{tabular}
\quad 
\begin{tabular}{c|cccccccc}
$\AA_3$&$1$&$2$&$3$&$4$&$5$&$6$&$7$&$8$\\
\hline
$1$&$2$&$4$&$6$&$8$&$2$&$4$&$6$&$8$\\
$2$&$3$&$4$&$7$&$8$&$3$&$4$&$7$&$8$\\
$3$&$4$&$8$&$4$&$8$&$4$&$8$&$4$&$8$\\
$4$&$5$&$6$&$7$&$8$&$5$&$6$&$7$&$8$\\
$5$&$6$&$8$&$6$&$8$&$6$&$8$&$6$&$8$\\
$6$&$7$&$8$&$7$&$8$&$7$&$8$&$7$&$8$\\
$7$&$8$&$8$&$8$&$8$&$8$&$8$&$8$&$8$\\
$8$&$1$&$2$&$3$&$4$&$5$&$6$&$7$&$8$\\
\end{tabular}\\
\def\arrayitem{5mm}
\smaller\begin{tabular}{c|c}
\VR(0,1.5)$\AA_4$&\SM1\SM2\SM3\SM4\SM5\SM6\SM7\SM8\SM9\SM{10}\SM{11}\SM{12}\SM{13}\SM{14}\SM{15}\SM{16}\\
\hline
\VR(3,0)\SM1&\SM2\SM{12}\SM{14}\SM{16}\SM2\SM{12}\SM{14}\SM{16}\SM2\SM{12}\SM{14}\SM{16}\SM2\SM{12}\SM{14}\SM{16}\\
\SM2&\SM3\SM{12}\SM{15}\SM{16}\SM3\SM{12}\SM{15}\SM{16}\SM3\SM{12}\SM{15}\SM{16}\SM3\SM{12}\SM{15}\SM{16}\\
\SM3&\SM4\SM8\SM{12}\SM{16}\SM4\SM8\SM{12}\SM{16}\SM4\SM8\SM{12}\SM{16}\SM4\SM8\SM{12}\SM{16}\\
\SM4&\SM5\SM6\SM7\SM8\SM{13}\SM{14}\SM{15}\SM{16}\SM5\SM6\SM7\SM8\SM{13}\SM{14}\SM{15}\SM{16}\\
\SM5&\SM6\SM8\SM{14}\SM{16}\SM6\SM8\SM{14}\SM{16}\SM6\SM8\SM{14}\SM{16}\SM6\SM8\SM{14}\SM{16}\\
\SM6&\SM7\SM8\SM{15}\SM{16}\SM7\SM8\SM{15}\SM{16}\SM7\SM8\SM{15}\SM{16}\SM7\SM8\SM{15}\SM{16}\\
\SM7&\SM8\SM{16}\SM8\SM{16}\SM8\SM{16}\SM8\SM{16}\SM8\SM{16}\SM8\SM{16}\SM8\SM{16}\SM8\SM{16}\\
\SM8&\SM9\SM{10}\SM{11}\SM{12}\SM{13}\SM{14}\SM{15}\SM{16}\SM9\SM{10}\SM{11}\SM{12}\SM{13}\SM{14}\SM{15}\SM{16}\\
\SM9&\SM{10}\SM{12}\SM{14}\SM{16}\SM{10}\SM{12}\SM{14}\SM{16}\SM{10}\SM{12}\SM{14}\SM{16}\SM{10}\SM{12}\SM{14}\SM{16}\\
\SM{10}&\SM{11}\SM{12}\SM{15}\SM{16}\SM{11}\SM{12}\SM{15}\SM{16}\SM{11}\SM{12}\SM{15}\SM{16}\SM{11}\SM{12}\SM{15}\SM{16}\\
\SM{11}&\SM{12}\SM{16}\SM{12}\SM{16}\SM{12}\SM{16}\SM{12}\SM{16}\SM{12}\SM{16}\SM{12}\SM{16}\SM{12}\SM{16}\SM{12}\SM{16}\\
\SM{12}&\SM{13}\SM{14}\SM{15}\SM{16}\SM{13}\SM{14}\SM{15}\SM{16}\SM{13}\SM{14}\SM{15}\SM{16}\SM{13}\SM{14}\SM{15}\SM{16}\\
\SM{13}&\SM{14}\SM{16}\SM{14}\SM{16}\SM{14}\SM{16}\SM{14}\SM{16}\SM{14}\SM{16}\SM{14}\SM{16}\SM{14}\SM{16}\SM{14}\SM{16}\\
\SM{14}&\SM{15}\SM{16}\SM{15}\SM{16}\SM{15}\SM{16}\SM{15}\SM{16}\SM{15}\SM{16}\SM{15}\SM{16}\SM{15}\SM{16}\SM{15}\SM{16}\\
\SM{15}&\SM{16}\SM{16}\SM{16}\SM{16}\SM{16}\SM{16}\SM{16}\SM{16}\SM{16}\SM{16}\SM{16}\SM{16}\SM{16}\SM{16}\SM{16}\SM{16}\\
\SM{16}&\SM1\SM2\SM3\SM4\SM5\SM6\SM7\SM8\SM9\SM{10}\SM{11}\SM{12}\SM{13}\SM{14}\SM{15}\SM{16}
\end{tabular}
\end{gather*}
\caption{\sf\smaller The first five Laver tables; we read for instance the values $\per_0(1) = 1$, $\per_1(1) = \per_2(1) = 2$; $\per_3(1) = \per_4(1) = 4$: the periods of the first rows make a non-decreasing sequence; the tables make an inverse system under projection mod~$2^\nn$: for instance, taking mod~$8$ the four values that occur in the first row of~$\AA_4$, namely $2, 12, 14, 16$, yields $2, 4, 6, 8$, which are the $4$~values that occur in the first row of~$\AA_3$; taking the latter mod~$4$ then gives $2, 4$ repeated twice, hence the two values that occur in the first row of~$\AA_2$, \emph{etc}.}
\label{T:LaverTables}
\end{table}

The way Richard Laver discovered the tables is remarkable, and, together with the orderability of braid groups, it is arguably one of the most interesting applications of large cardinals ideas in algebra~\cite{Dfp}.

In recent Set Theory, large cardinal axioms play an important r\^ole as natural axioms that can be added to the basic Zermelo-Fraenkel system to enhance its logical power. A number of such axioms state the existence of certain elementary embeddings, that is, of injective maps that preserve every notion that is first-order definable from the membership relation. One of the most simple such statements, Axiom~$\AxI3$, asserts the existence of a (nontrivial, that is, non-bijective) elementary embedding from a limit rank~$V_\lambda$ to itself~\cite{SRK, Kan}. The point here is that, if such an object exists, then the family~$\EEE_\lambda$ of all such elementary embeddings of~$V_\lambda$ to itself equipped with the binary operation $j[k] := \bigcup_{\alpha < \lambda} j(k \res V_\alpha)$ is an LD-system, that is, the relation $\ii[\jj[\kk]] = \ii[\jj][\ii[\kk]]$ holds. It has been known since the 1980's that the algebraic structures~$(\EEE_\lambda, [\,])$ have nontrivial properties~\cite{Dem}. Investigating them since the time of~\cite{Lva}, Laver was naturally led to introducing their quotients obtained by cutting the graphs of the elementary embeddings at some level. Laver proved that, for every~$\nn$, cutting at the level of what is called the $2^\nn$th critical ordinal yields a finite quotient with $2^\nn$~elements and that the latter enjoys the properties listed in Theorem~\ref{T:Tables}. What is remarkable here is that, once the definition of Theorem~\ref{T:Tables}\ITEM1 has been isolated, the existence of the tables and the basic properties listed in~\ITEM1--\ITEM4 can be established directly, without appealing to elementary embeddings and, therefore, they do not require any large cardinal assumption. By contrast, for the properties listed in~\ITEM5, no direct combinatorial proof has been found so far and, therefore, one cannot assert them without assuming the (unprovable) existence of an elementary embedding of the needed type, which is precisely the (strong) large cardinal axiom~$\AxI3$. We refer to Chapters~X, XII, and~XIII of~\cite{Dgd} for details.

The algebraic investigation of Laver tables was pursued in two directions. The first one is the study of general (finite) LD-systems, which proved to be an intricate question. As shown by A.\,Dr\'apal in~\cite{DraCyc, DraAlg, DraGro}, the global result is that Laver tables are the fundamental objects when one considers finite LD-systems with one generator: every such LD-system can be obtained from Laver tables by means of various transformations, see~\cite[Section~X.2]{Dgd} for precise statements, and also the recent preprint~\cite{Sme}, which offers a simplified description in a restricted situation. Summarizing, we can say that the Laver tables play, in the world of selfdistributivity, the same r\^ole as the one played by the cyclic groups $\ZZZZ/\pp\ZZZZ$ in the world of associativity. 

The second direction of research was to try to discard the large cardinal assumption in Theorem~\ref{T:Tables}\ITEM5---or, contrariwise, to prove that it is necessary. So far, only partial results have been obtained. In the direction of eliminating the axiom, A.\,Dr\'apal established in~\cite{Dra1, Dra2, Dra3} the first three steps of a program which, if completed, would show that $\per_\nn(1)$ tends to infinity with~$\nn$. The combinatorial complexity increases so fast that the problem was then abandoned. In the other direction, it was proved by R.\,Dougherty and T.\,Jech \cite{DoJ} that $\per_\nn(1)$ tends to infinity (if it does) at least as slow as the functional inverse of the Ackermann function, implying that its divergence cannot be proved in Primitive Recursive Arithmetic. This however says nothing for Peano Arithmetic nor, {\it a fortiori} for the Zermelo--Fraenkel system. Note that, in contradistinction with the properties of free LD-systems, in particular the irreflexivity property of~\cite{Lvb} and~\cite{Dez}, that were first proved using Axiom~$\AxI3$ and subsequently without it, the properties of the LD-systems~$\EEE_\lambda$ used to establish that $\per_\nn(1)$ tends to infinity are not trivial from a set-theoretical point of view, relying on a deep result by J.\,Steel about extenders. This might explain why discarding the large cardinal axiom is more difficult here, see~\cite{Dgs} for details.

\subsection{The diagram colouring approach}
\label{SS:Colouring}

We now turn to possible applications of the Laver tables in low-dimensional topology, starting here with the principle of using selfdistributive structures to colour the strands of link or braid diagrams and its known implementations: in this subsections as in the next ones, we first recall the general principle (thus mentioning elements that are mostly standard in the topology community) and then consider the specific case of Laver tables. 

As a general preliminary remark, we would like to insist on the fact that, according to Theorem~\ref{T:Tables}, Laver tables are closely connected with free LD-systems, so, in some sense, with the most general LD-systems. By contrast, all racks and quandles that have been used so far in topology (see Definition~\ref{D:Rack} below) are, by very definition, quite far from being free LD-systems: for instance, every rack (here in its left-selfdistributive version) satisfies the law $(\xx \op \xx) \op \yy = \xx \op \yy$ and its operation is closely connected with the conjugacy operation of a group. This is not at all the case with general LD-systems, and with Laver tables in particular: in a sense, this is bad news as it may suggest that none of the existing tools will extend, but, in another sense, this is good news as this suggests that any possible application of the Laver tables has good chances to be really new---as was the application of free LD-systems to braid orderability twenty years ago.

\subsubsection*{The general principle}

In order to investigate embedded $1$-dimensional objects like knots, links, braids, one usually starts with diagrams similar to those of Figure~\ref{F:Diagrams}, which are seen as plane projections of curves embedded in~$\RRRR^3$, and the generic question is to recognize whether two diagrams represent ambient isotopic curves, that is, whether there exists a continuous deformation of the ambient space that takes the curves projecting to the first diagram to the curve projecting to the second diagram. 

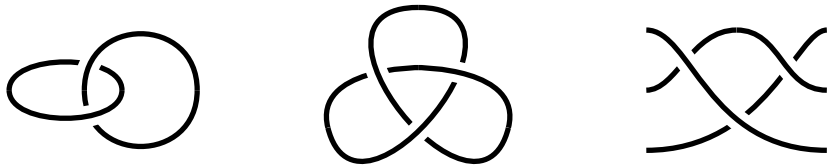
\begin{figure}[htb]
\begin{picture}(25,23)(0,0)
\psbezier[linewidth=2pt, border=3pt](0,10)(0,15)(15,15)(15,10)
\psbezier[linewidth=2pt, border=3pt](10,10)(10,0)(25,0)(25,10)
\psbezier[linewidth=2pt, border=3pt](10,10)(10,20)(25,20)(25,10)
\psbezier[linewidth=2pt, border=3pt](0,10)(0,5)(15,5)(15,10)
\end{picture}
\hspace{15mm}
\begin{picture}(25,20)(0,3)
\psbezier[linewidth=2pt, border=3pt](0,8)(-2,16)(12,16)(12,16)
\psbezier[linewidth=2pt, border=3pt](12,24)(-6,24)(20,-8)(24,8)
\psbezier[linewidth=2pt, border=3pt](0,8)(4,-8)(30,24)(12,24)
\psbezier[linewidth=2pt, border=3pt](12,16)(12,16)(26,16)(24,8)
\end{picture}
\hspace{15mm}
\begin{picture}(24,20)(0,-2)
\psbezier[linewidth=2pt,border=3pt](0,8)(4,8)(6,16)(12,16)
\psbezier[linewidth=2pt,border=3pt](0,0)(16,0)(20,16)(24,16)
\psbezier[linewidth=2pt,border=3pt](12,16)(18,16)(18,8)(24,8)
\psbezier[linewidth=2pt,border=3pt](0,16)(6,16)(8,0)(24,0)
\end{picture}
\caption{\sf\smaller Three diagrams, respectively representing the Hopf link (two circles embedded in~$\RRRR^3$), the trefoil knot (one embedded circle), and Garside's fundamental braid~$\Delta_3$.}
\label{F:Diagrams}
\end{figure}

A classical result---see for instance \cite{Bir}, \cite{BuZ}, or~\cite{Kau}---asserts that two diagrams represent ambient isotopic links if and only if they can be transformed into one another by means of the three types of Reidemeister moves displayed in Figure~\ref{F:Moves}.

\begin{figure}[htb]
\begin{picture}(51,15)(0,0)
\psbezier[linewidth=2pt, border=3pt](0,0)(10,0)(15,10)(8,10)
\psbezier[linewidth=2pt, border=3pt](8,10)(1,10)(6,0)(16,0)
\put(17,4){$\sim$}
\psline[linewidth=2pt, border=3pt](21,0)(35,0)
\put(36,4){$\sim$}
\psbezier[linewidth=2pt, border=3pt](48,10)(41,10)(46,0)(56,0)
\psbezier[linewidth=2pt, border=3pt](40,0)(50,0)(55,10)(48,10)
\put(0,13){type I:}
\end{picture}
\hfill
\begin{picture}(65,16)(0,-1)
\psbezier[linewidth=2pt, border=3pt](0,0)(6,0)(6,8)(10,8)
\psbezier[linewidth=2pt, border=3pt](10,8)(14,8)(14,0)(20,0)
\psbezier[linewidth=2pt, border=3pt](0,8)(6,8)(6,0)(10,0)
\psbezier[linewidth=2pt, border=3pt](10,0)(14,0)(14,8)(20,8)
\put(21,3){$\sim$}
\psline[linewidth=2pt, border=3pt](25,0)(39,0)
\psline[linewidth=2pt, border=3pt](25,8)(39,8)
\put(40,3){$\sim$}
\psbezier[linewidth=2pt, border=3pt](45,8)(51,8)(51,0)(55,0)
\psbezier[linewidth=2pt, border=3pt](55,0)(59,0)(59,8)(65,8)
\psbezier[linewidth=2pt, border=3pt](45,0)(51,0)(51,8)(55,8)
\psbezier[linewidth=2pt, border=3pt](55,8)(59,8)(59,0)(65,0)
\put(0,12){type II:}
\end{picture}
\hspace{15mm}
\begin{picture}(55,27)(0,-2)
\psbezier[linewidth=2pt,border=3pt](0,8)(4,8)(6,16)(12,16)
\psbezier[linewidth=2pt,border=3pt](0,0)(16,0)(20,16)(24,16)
\psbezier[linewidth=2pt,border=3pt](12,16)(18,16)(18,8)(24,8)
\psbezier[linewidth=2pt,border=3pt](0,16)(6,16)(8,0)(24,0)
\put(26,7){$\sim$}
\psbezier[linewidth=2pt,border=3pt](30,0)(36,0)(40,16)(54,16)
\psbezier[linewidth=2pt,border=3pt](42,0)(48,0)(48,8)(54,8)
\psbezier[linewidth=2pt,border=3pt](30,16)(42,16)(48,0)(54,0)
\psbezier[linewidth=2pt,border=3pt](30,8)(34,8)(36,0)(42,0)
\put(0,19){type III:}
\end{picture}
\caption{\sf\smaller Reidemeister moves: two diagrams represent ambient isotopic figures if and only if they can be transformed into one another using a finite sequence of such moves.}
\label{F:Moves}
\end{figure}
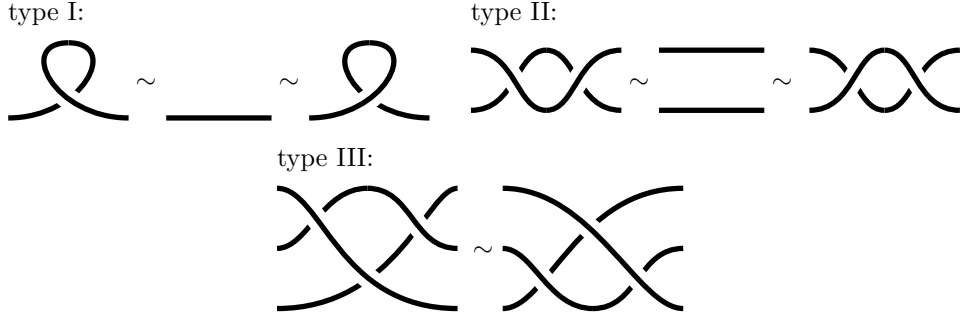

The idea of strand colouring is then natural: assuming that an auxiliary set~$\SS$ (the ``colours'') has been fixed, we attach to each arc in the considered diagram a label from~$\SS$, with the aim of extracting information about the isotopy class of the considered diagram. If the colours do not change when two strands cross, the only piece of information that can be extracted is the number of connected components in the case of a (closed) link diagram, or the permutation associated with the braid in the case of an (open) braid diagram. Things become more interesting when colours are allowed to change at crossings. The simplest case is when the strands are oriented and, when two strands cross, only the colour of the top arc may change and its new colour only depends on the colours of the two arcs involved in the crossing and of the orientation of the latter. This amounts to assuming that the set of colours~$\SS$ is equipped with two binary operations~$\op, \opbar$ and the colours obey the rules
\begin{equation}
\label{E:Rules}
\VR(6,6)\begin{picture}(55,0)(0,4)
\psbezier[linewidth=2pt,border=3pt](0,0)(4,0)(4,8)(8,8)
\psbezier[linewidth=2pt,border=3pt](0,8)(4,8)(4,0)(8,0)
\pcline[linewidth=2pt,border=3pt](-1,0)(0,0)
\pcline[linewidth=2pt,border=3pt](-1,8)(0,8)
\pcline[linewidth=2pt,border=3pt]{->}(8,8)(10,8)
\pcline[linewidth=2pt,border=3pt]{->}(8,0)(10,0)
\put(-4,7.5){$\bb$}
\put(-4,-0.5){$\aa$}
\put(11.5,7.5){$\aa$}
\put(11.5,-0.5){$\aa \op \bb$}
\put(23,3){and}
\psbezier[linewidth=2pt,border=3pt](40,8)(44,8)(44,0)(48,0)
\psbezier[linewidth=2pt,border=3pt](40,0)(44,0)(44,8)(48,8)
\psline[linewidth=2pt,border=3pt](39,0)(40,0)
\psline[linewidth=2pt,border=3pt](39,8)(40,8)
\psline[linewidth=2pt,border=3pt]{->}(48,8)(50,8)
\psline[linewidth=2pt,border=3pt]{->}(48,0)(50,0)
\put(36,7.5){$\bb$}
\put(36,-0.5){$\aa$}
\put(51.5,7.5){$\aa \opbar \bb$}
\put(51.5,-0.5){$\bb\quad.$}
\end{picture}
\end{equation}
Now, in order to possibly extract information about the isotopy class of a diagram, we have to request that the colours do not change when an isotopy is performed or, more exactly, that an admissible colouring is mapped to (another) admissible colouring with the same output. Both in the case of closed diagrams (knots and links) and in the case of open diagrams (braids), this amounts to requiring that, when a Reidemeister move is performed and some input colours are applied to the (oriented) strands, then the output colours are not changed. This immediately translates into algebraic constraints for the operations~$\op$ and~$\opbar$.

\begin{lemm}[Joyce \cite{Joy}]
\label{L:Translation}
Assume that $\op$ and $\opbar$ are binary operations on~$\SS$. Then $(\SS, \op, \opbar)$-colourings are invariant under Reidemeister moves if and only if $(\SS, \op, \opbar)$ obeys the following laws:
\begin{align}
\label{E:RI}
&\text{type~I:} 
&&\xx \op \xx = \xx \opbar \xx = \xx;\\
\label{E:RII}
&\text{type~II:} 
&&\xx \op (\xx \opbar \yy) = \xx \opbar (\xx \op \yy) = \yy;\\
\label{E:RIII}
&\text{type~III:} 
&&\xx \op' (\yy \op'' \zz) = (\xx \op' \yy) \op'' (\xx \op' \zz)\ 
\text{for $\op', \op''$ ranging in $\{\op, \opbar\}$}.
\end{align}
\end{lemm}

The proof is given in Figure~\ref{F:Translation}.

\begin{figure}[htb]
\begin{picture}(71,15)(0,0)
\put(-20,3){type I:}
\psbezier[linewidth=2pt, border=3pt](0,0)(10,0)(15,10)(8,10)
\psbezier[linewidth=2pt, border=3pt](8,10)(1,10)(6,0)(16,0)
\psline[linewidth=2pt,border=3pt]{->}(16,0)(18,0)
\put(-3,-1){$\aa$}
\put(2,8){$\aa$}
\put(19,-1){$\aa \op \aa$}
\put(28,4){$\sim$}
\psline[linewidth=2pt, border=3pt]{->}(36,0)(52,0)
\put(33,-1){$\aa$}
\put(53,-1){$\aa$}
\put(57,4){$\sim$}
\psbezier[linewidth=2pt, border=3pt](73,10)(66,10)(71,0)(81,0)
\psbezier[linewidth=2pt, border=3pt](65,0)(75,0)(80,10)(73,10)
\psline[linewidth=2pt,border=3pt]{->}(81,0)(83,0)
\put(62,-1){$\aa$}
\put(67,8){$\aa$}
\put(84,-1){$\aa \opbar \aa$}
\end{picture}

\begin{picture}(86,18)(0,0)
\put(-18,3){type II:}
\psbezier[linewidth=2pt, border=3pt](0,0)(6,0)(6,8)(10,8)
\psbezier[linewidth=2pt, border=3pt](10,8)(14,8)(14,0)(20,0)
\psbezier[linewidth=2pt, border=3pt](0,8)(6,8)(6,0)(10,0)
\psbezier[linewidth=2pt, border=3pt](10,0)(14,0)(14,8)(20,8)
\psline[linewidth=2pt,border=3pt]{->}(20,0)(22,0)
\psline[linewidth=2pt,border=3pt]{->}(20,8)(22,8)
\put(-3,-1){$\aa$}
\put(-3,7){$\bb$}
\put(23,-1){$\aa$}
\put(23,7){$\aa {\opbar} (\aa {\op} \bb)$}
\put(7,-3){$\aa {\op} \bb$}
\put(9,9){$\aa$}
\put(36,3){$\sim$}
\psline[linewidth=2pt, border=3pt]{->}(43,0)(59,0)
\psline[linewidth=2pt, border=3pt]{->}(43,8)(59,8)
\put(40,-1){$\aa$}
\put(40,7){$\bb$}
\put(60,-1){$\aa$}
\put(60,7){$\bb$}
\put(63,3){$\sim$}
\psbezier[linewidth=2pt, border=3pt](70,8)(76,8)(76,0)(80,0)
\psbezier[linewidth=2pt, border=3pt](80,0)(84,0)(84,8)(90,8)
\psbezier[linewidth=2pt, border=3pt](70,0)(76,0)(76,8)(80,8)
\psbezier[linewidth=2pt, border=3pt](80,8)(84,8)(84,0)(90,0)
\psline[linewidth=2pt,border=3pt]{->}(90,0)(92,0)
\psline[linewidth=2pt,border=3pt]{->}(90,8)(92,8)
\put(67,-1){$\aa$}
\put(67,7){$\bb$}
\put(93,7){$\bb$}
\put(93,-1){$\bb {\op} (\bb {\opbar} \aa)$}
\put(77,9){$\bb {\opbar} \aa$}
\put(79,-3){$\bb$}
\end{picture}

\begin{picture}(75,25)(0,0)
\put(-22,7){type III$_{\scriptscriptstyle++}$:}
\psbezier[linewidth=2pt,border=3pt](0,8)(4,8)(6,16)(12,16)
\psbezier[linewidth=2pt,border=3pt](0,0)(16,0)(20,16)(24,16)
\psbezier[linewidth=2pt,border=3pt](12,16)(18,16)(18,8)(24,8)
\psbezier[linewidth=2pt,border=3pt](0,16)(6,16)(8,0)(24,0)
\psline[linewidth=2pt,border=3pt]{->}(24,0)(26,0)
\psline[linewidth=2pt,border=3pt]{->}(24,8)(26,8)
\psline[linewidth=2pt,border=3pt]{->}(24,16)(26,16)
\put(-3,-1){$\aa$}
\put(-3,7){$\bb$}
\put(-3,15){$\cc$}
\put(27,-1){$\aa {\op} (\bb {\op} \cc)$}
\put(27,7){$\aa{\op}\bb$}
\put(27,15){$\aa$}
\put(9,9){$\bb{\op}\cc$}
\put(11,17.5){$\bb$}
\put(41,7){$\sim$}
\psbezier[linewidth=2pt,border=3pt](50,0)(56,0)(62,16)(74,16)
\psbezier[linewidth=2pt,border=3pt](62,0)(68,0)(68,8)(74,8)
\psbezier[linewidth=2pt,border=3pt](50,16)(62,16)(68,0)(74,0)
\psbezier[linewidth=2pt,border=3pt](50,8)(54,8)(56,0)(62,0)
\psline[linewidth=2pt,border=3pt]{->}(74,0)(76,0)
\psline[linewidth=2pt,border=3pt]{->}(74,8)(76,8)
\psline[linewidth=2pt,border=3pt]{->}(74,16)(76,16)
\put(47,-1){$\aa$}
\put(47,7){$\bb$}
\put(47,15){$\cc$}
\put(77,-1){$(\aa {\op}\bb) {\op} (\aa {\op} \cc)$}
\put(77,7){$\aa{\op}\bb$}
\put(77,15){$\aa$}
\put(59,-3){$\aa{\op}\bb$}
\put(59,5){$\aa$}
\put(65,8.5){$\aa{\op}\cc$}
\end{picture}

\begin{picture}(75,27)(0,-2)
\put(-22,7){type III$_{\scriptscriptstyle-+}$:}
\psbezier[linewidth=2pt,border=3pt](0,16)(6,16)(8,0)(24,0)
\psbezier[linewidth=2pt,border=3pt](0,8)(4,8)(6,16)(12,16)
\psbezier[linewidth=2pt,border=3pt](0,0)(16,0)(20,16)(24,16)
\psbezier[linewidth=2pt,border=3pt](12,16)(18,16)(18,8)(24,8)
\psline[linewidth=2pt,border=3pt]{->}(24,0)(26,0)
\psline[linewidth=2pt,border=3pt]{->}(24,8)(26,8)
\psline[linewidth=2pt,border=3pt]{->}(24,16)(26,16)
\put(-3,-1){$\aa$}
\put(-3,7){$\bb$}
\put(-3,15){$\cc$}
\put(27,-1){$\cc$}
\put(27,7){$(\cc{\opbar}\aa){\op}(\cc{\opbar}\bb)$}
\put(27,15){$\cc{\opbar}\aa$}
\put(9,17.5){$\cc{\opbar}\bb$}
\put(16,5){$\cc{\opbar}\aa$}
\put(9.5,8){$\cc$}
\put(47,7){$\sim$}
\psbezier[linewidth=2pt,border=3pt](55,16)(67,16)(73,0)(79,0)
\psbezier[linewidth=2pt,border=3pt](55,0)(61,0)(67,16)(79,16)
\psbezier[linewidth=2pt,border=3pt](67,0)(73,0)(73,8)(79,8)
\psbezier[linewidth=2pt,border=3pt](55,8)(59,8)(61,0)(67,0)
\psline[linewidth=2pt,border=3pt]{->}(79,0)(81,0)
\psline[linewidth=2pt,border=3pt]{->}(79,8)(81,8)
\psline[linewidth=2pt,border=3pt]{->}(79,16)(81,16)
\put(52,-1){$\aa$}
\put(52,7){$\bb$}
\put(52,15){$\cc$}
\put(82,-1){$\cc$}
\put(82,7){$\cc{\opbar}(\aa{\op}\bb)$}
\put(82,15){$\cc{\opbar}\aa$}
\put(64,-3){$\aa{\op}\bb$}
\put(64,5){$\aa$}
\put(70,8.5){$\cc$}
\end{picture}
\caption{\sf\smaller Translation of invariance under Reidemeister moves into algebraic laws for the colourings: in the case of Reidemeister~III, four orientations are possible, of which only two are displayed; the other two are similar and correspond to the last two combinations of~$\op$ and~$\opbar$.}
\label{F:Translation}
\end{figure}
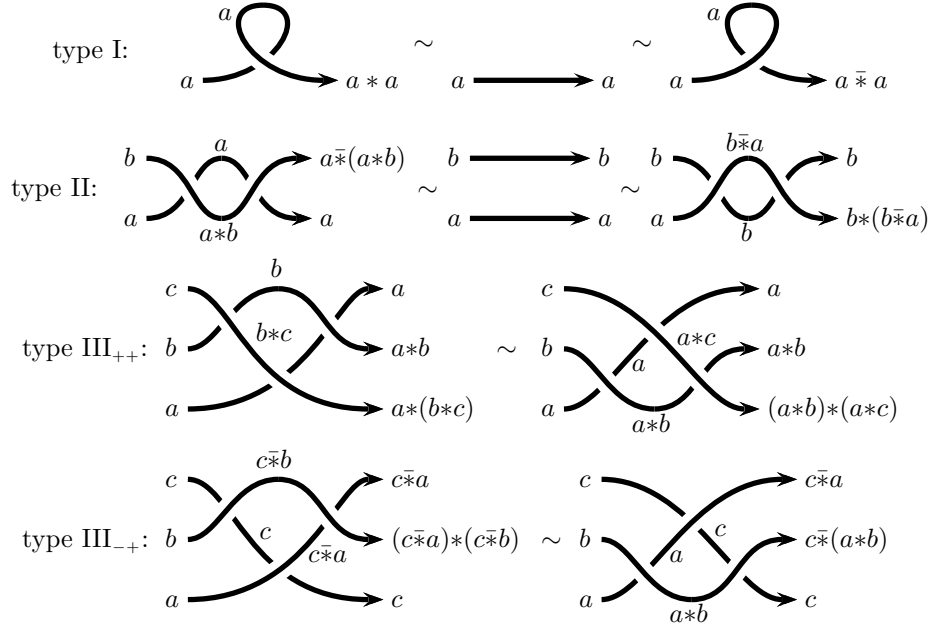

One is thus led to considering the structures involving two binary operations obeying the laws listed in~\eqref{E:RI}--\eqref{E:RIII}. An easy observation is that, in such structures, each operation determines the other and that the four laws of~\eqref{E:RIII} reduce to a single law.

\begin{lemm}
A structure $(\SS, \op, \opbar)$ obeys~\eqref{E:RII} if and only if the left-translations of~$(\SS, \op)$ are bijective and, for all $\aa, \bb$ in~$\SS$, one has
\begin{equation}
\label{E:Inverse}
\aa \opbar \bb = \text{the unique element $\cc$ of~$\SS$ satisfying $\aa \op \cc = \bb$}.
\end{equation}
In this case, the laws of \eqref{E:RI} (\resp \eqref{E:RIII}) are satisfied if and only if $(\SS, \op)$ obeys 
\begin{gather}
\label{E:Idem}
\xx \op \xx = \xx\\
\label{E:LDBis}
(\resp \xx \op (\yy \op \zz) = (\xx \op \yy) \op (\xx \op \zz)).
\end{gather}
\end{lemm}

We skip the (easy) verification, which can be found for instance in~\cite{Dhy}. It is then natural to introduce a terminology for those LD-systems that satisfy the additional laws of Lemma~\ref{L:Translation}.

\begin{defi}\cite{FeR, Joy}
\label{D:Rack}
An LD-system~$(\SS, \op)$ in which all left-translations are bijective---or, equivalently, a structure $(\SS, \op, \opbar)$ where \eqref{E:RII} and~\eqref{E:RIII} are obeyed---is called a \emph{rack}. An idempotent rack, that is, a rack satisfying~\eqref{E:Idem}, is called a \emph{quandle}.
\end{defi}

\begin{rema}
Various names appear for the above structures in literature. For instance, racks are called \emph{automorphic sets} in the early source~\cite{Bri}, whereas the terms \emph{LD-quasigroup} and \emph{LDI-quasigroups} would be coherent with the standards the algebra community for rack and quandle respectively. More importantly, the most common convention in the topology community is to appeal to the opposite operations, namely to define colourings by the rules
\begin{equation}
\label{E:RulesOpp}
\VR(6,6)\begin{picture}(55,0)(0,4)
\psbezier[linewidth=2pt,border=3pt](0,0)(4,0)(4,8)(8,8)
\psbezier[linewidth=2pt,border=3pt](0,8)(4,8)(4,0)(8,0)
\pcline[linewidth=2pt,border=3pt](-1,0)(0,0)
\pcline[linewidth=2pt,border=3pt](-1,8)(0,8)
\pcline[linewidth=2pt,border=3pt]{->}(8,8)(10,8)
\pcline[linewidth=2pt,border=3pt]{->}(8,0)(10,0)
\put(-4,7.5){$\bb$}
\put(-4,-0.5){$\aa$}
\put(11.5,7.5){$\aa \opt \bb$}
\put(11.5,-0.5){$\bb$}
\put(23,3){and}
\psbezier[linewidth=2pt,border=3pt](40,8)(44,8)(44,0)(48,0)
\psbezier[linewidth=2pt,border=3pt](40,0)(44,0)(44,8)(48,8)
\psline[linewidth=2pt,border=3pt](39,0)(40,0)
\psline[linewidth=2pt,border=3pt](39,8)(40,8)
\psline[linewidth=2pt,border=3pt]{->}(48,8)(50,8)
\psline[linewidth=2pt,border=3pt]{->}(48,0)(50,0)
\put(36,7.5){$\bb$}
\put(36,-0.5){$\aa$}
\put(51.5,7.5){$\aa$}
\put(51.5,-0.5){$\aa \optbar \bb\quad.$}
\end{picture}
\end{equation}
The effect of these conventions is to replace left-selfdistributivity with its right counterpart (RD) $(\xx \opt \yy) \opt \zz = (\xx \opt \zz) \opt (\yy \opt \zz)$ everywhere and, of course, to consider right-translations in the definition of a rack. Because of our specific interest in the Laver tables here, we shall stick to \eqref{E:Rules} and the law~(LD) here. To avoid ambiguity, we use $\op$ and~$\opbar$ as a generic notation for LD-operations, thus keeping $\opt$ and $\optbar$ for RD-operations. Of course, the transpose~$\widetilde\AA_\nn$ of the Laver table~$\AA_\nn$ is an RD-system with $2^\nn$~elements.
\end{rema}

\subsubsection*{Using colourings: case of braids}

Using diagram colourings takes different forms according to whether the considered diagram is open (braid diagram) or closed (link diagram). We begin with the case of braids.

An \emph{$\mm$-strand geometric braid} is a family of $\mm$ open curves embedded in $\RRRR^2 \times [0, 1]$ such that the family of initial points is $\{(0, \ii, 0) \mid \ii = 1 \wdots \mm\}$, the family of final points is $\{(0, \ii, 1) \mid \ii = 1 \wdots \mm\}$, and, for every~$\tt$, the intersection with the plane $\zz = \tt$ consists of $\mm$~points exactly. A \emph{braid} is an isotopy class of geometric braids, referring here to isotopies of $\RRRR^2 \times [0, 1]$ that leave the planes $\RRRR^2 \times \{0\}$ and $\RRRR^2 \times \{1\}$ fixed. Projecting a geometric braid on the plane $\xx = 0$ gives a diagram like the one on the right of Figure~\ref{F:Diagrams}: the specificity is that there exists a fixed orientation so that the diagrams consists of $\mm$~arcs going from the line $\xx = 0$ to the line $\xx = 1$ in such a way that the $\xx$~coordinate keeps increasing (no U-turn). 

{\rightskip30mm
Concatenating $\mm$~strand geometric braids induces (after rescaling) a well-defined product on $\mm$~strand braids, which turns to provide a group structure as, by Reidemeister moves of type~II, the concatenation of a geometric braid and its image in a vertical mirror is isotopic to the trivial braid, a collection of horizontal segments. Calling $\sig\ii$ the (class of the geometric) braid that projects as shown on the right, one easily shows that the group~$\BR\mm$ of all $\mm$-strand braids is generated by $\sig1 \wdots \sig{\mm-1}$.\par
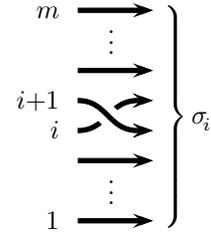
\hfill
\begin{picture}(0,0)(16,-7)
\pcline[linewidth=2pt,border=3pt]{->}(0,0)(10,0)
\tlput{$1$}
\pcline[linewidth=2pt,border=3pt]{->}(0,8)(10,8)
\pcline[linewidth=2pt,border=3pt](0,12)(0,12)
\tlput{$\ii$}
\psbezier[linewidth=2pt,border=3pt](0,12)(4,12)(4,16)(8,16)
\pcline[linewidth=2pt,border=3pt](0,16)(0,16)
\tlput{$\ii{+}1$}
\psbezier[linewidth=2pt,border=3pt](0,16)(4,16)(4,12)(8,12)
\pcline[linewidth=2pt,border=3pt]{->}(8,12)(10,12)
\pcline[linewidth=2pt,border=3pt]{->}(8,16)(10,16)
\pcline[linewidth=2pt,border=3pt]{->}(0,20)(10,20)
\pcline[linewidth=2pt,border=3pt]{->}(0,28)(10,28)
\tlput{$\mm$}
\put(4,2.2){$\vdots$}
\put(4,22.2){$\vdots$}
\put(11,13){$\left.\vbox to 16mm{}\right\} \sig\ii$}
\end{picture}}
\vspace{-3.5mm}

It was then proved by E.\,Artin in~\cite{ArtHamburg, Art} that the group~$\BR\mm$ admits the presentation
\begin{equation}
\label{E:BraidPres}
\bigg\langle \sig1 \wdots \sig{n-1} \ \bigg\vert\ 
\begin{matrix}
\sig\ii \sig\jj = \sig\jj \sig\ii 
&\text{for} &\vert \ii-\jj\vert\ge 2\\
\sig\ii \sig\jj \sig\ii = \sig\jj \sig\ii \sig\jj 
&\text{for} &\vert\ii-\jj\vert = 1
\end{matrix}
\ \bigg\rangle,
\end{equation}
and, by F.A.\,Garside in~\cite{Gar}, that the submonoid~$\BP\mm$ of~$\BR\mm$ generated by~$\sig1 \wdots \sig{\mm-1}$ admits, as a monoid, the presentation~\eqref{E:BraidPres}. The elements of the monoid~$\BP\mm$ are called \emph{positive} $\mm$-strand braids. By definition, they can be represented by braid diagrams in which all crossings have the same orientation.

Using a fixed structure $(\SS, \op, \opbar)$ to colour the strands of an $\SS$-strand braid diagram using the rules~\eqref{E:Rules} provides a map from~$\SS^\mm$ to itself, namely the map that associates the sequence of output colours to the sequence of input colours. By Lemma~\ref{L:Translation}, this map is isotopy-invariant whenever $(\SS, \op)$ is a quandle. Actually, due to the definition of braids with U-turns forbidden, Reidemeister moves of type~I are impossible, and it is sufficient to use racks that need not be quandles. Similarly, when one considers positive braids, Reidemeister moves of type~II are impossible, and using general LD-systems becomes possible. As, by very definition, colourings are compatible with the product of braids, Lemma~\ref{L:Translation} takes the form:

\begin{lemm}[Brieskorn \cite{Bri}]
\label{L:Action}
\ITEM1 Assume that $(\SS, \op)$ is a rack. Then putting
\begin{align} 
\label{E:Action}
(\aa_1 \wdots \aa_\mm) \act \sig\ii 
&= (\aa_1 \wdots \aa_{\ii-1}, \aa_\ii \op \aa_{\ii+1}, \aa_\ii, \aa_{\ii+2} \wdots \aa_\mm),\\
\label{E:ActionInv}
(\aa_1 \wdots \aa_\mm) \act \siginv\ii 
&= (\aa_1 \wdots \aa_{\ii-1}, \aa_{\ii+1}, \aa_\ii \opbar \aa_{\ii+1}, \aa_{\ii+2} \wdots \aa_\mm)
\end{align}
defines an action (on the right) of the group~$\BR\mm$ on~$\SS^\mm$.

\ITEM2 Assume that $(\SS, \op)$ is an LD-system. Then \eqref{E:Action} defines an action of the monoid~$\BP\mm$ on~$\SS^\mm$.
\end{lemm}

The action of Lemma~\ref{L:Action} is called the \emph{Hurwitz action}. Using classical examples of racks then leads to no less classical examples of braid invariants. For instance, considering $\ZZZZ$ equipped with the operation $\xx \op \yy = \yy+1$ leads to the augmentation homomorphism from~$\BR\mm$ to~$(\ZZZZ, +)$, whereas considering a $\ZZZZ[t, t\inv]$-module equipped with the binary operations $\xx \op \yy = (1 - t)\xx  + t \yy$ leads to a linear representation of~$\BR\mm$ into $\GL_n(\ZZZZ[t, t\inv])$, the (unreduced) \emph{Burau representation} of~$\BR\mm$. Similarly, considering a rank~$\mm$ free group~$\FF_\mm$ equipped with the conjugacy operation $\xx \op \yy = \xx\yy\xx\inv$ leads to a (faithful) representation of~$\BR\mm$ in~$\Aut(\FF_\mm)$, the \emph{Artin representation}.

\subsubsection*{Using colourings: case of links and knots}

An (oriented) \emph{$\mm$-component geometric link} is a family of $\mm$ disjoint closed curves embedded in $\RRRR^3$. A \emph{link} is an isotopy class of geometric links, referring here to isotopies of $\RRRR^3$. Knots are links with one component. Projecting geometric links to a plane keeping track of the orientation of crossings and avoiding triple points and tangencies yields a link diagram. As already said, two diagrams represent the same link if and only if they can be transformed into each other by means of Reidemeister moves. 

At least two different approaches have been developed in order to use selfdistributive structures to construct link invariants via the colouring approach. Developed by D.\,Joyce \cite{Joy} and S.\,Matveev~\cite{Mat}, the first one consists in attaching to every diagram a specific quandle that will capture the topology of the link it represents: assuming that the considered link diagram is the closure~$\widehat\DD$ of an $\mm$-strand braid diagram~$\DD$ (see Figure~\ref{F:FundQuandle}), one uses $\mm$~letters $\aa_1 \wdots \aa_\mm$ to colour the input ends of the braid diagram, one propagates the colours throughout the diagram resulting in $\mm$ output colours $\tt_1 \wdots \tt_\mm$ which are formal combinations of $\aa_1 \wdots \aa_\mm$ by means of two formal operations~$\op, \opbar$, and one defines the \emph{fundamental quandle}~$\QQ_\DD$ to be the quandle that admits the presentation $\langle \aa_1 \wdots \aa_\mm \mid \tt_1 = \aa_1 \wdots \tt_\mm = \aa_\mm\rangle$. The quandle laws imply that $\QQ_\DD$ only depends on the isotopy class of~$\widehat\DD$, and it captures almost all topological information about the link represented by~$\DD$ as it is a complete invariant of the isotopy type up to a mirror symmetry. In practice, determining the fundamental quandle effectively is possible only in simple particular cases~\cite{NiePrz}, so one tends to consider more simple structures, typically quotients of the fundamental quandle like the Alexander quandle from which the Alexander polynomial can be read~\cite{FRSJames, FRSRack}. 

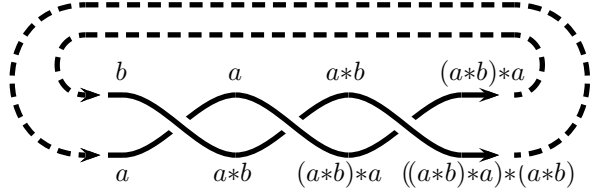
\begin{figure}[htb]
\begin{picture}(110,24)(0,-2)
\psline[linewidth=2pt,border=3pt](28,0)(30,0)
\psline[linewidth=2pt,border=3pt](28,8)(30,8)
\psbezier[linewidth=2pt,border=3pt](30,0)(35,0)(40,8)(45,8)
\psbezier[linewidth=2pt,border=3pt](30,8)(35,8)(40,0)(45,0)
\psbezier[linewidth=2pt,border=3pt](45,0)(50,0)(55,8)(60,8)
\psbezier[linewidth=2pt,border=3pt](45,8)(50,8)(55,0)(60,0)
\psbezier[linewidth=2pt,border=3pt](60,0)(65,0)(70,8)(75,8)
\psbezier[linewidth=2pt,border=3pt](60,8)(65,8)(70,0)(75,0)
\psline[linewidth=2pt,border=3pt]{->}(75,0)(80,0)
\psline[linewidth=2pt,border=3pt]{->}(75,8)(80,8)
\put(29,-3.5){$\aa$}
\put(29,10){$\bb$}
\put(42,-3.5){$\aa{\op}\bb$}
\put(44,10){$\aa$}
\put(53,-3.5){$(\aa{\op}\bb){\op}\aa$}
\put(57,10){$\aa{\op}\bb$}
\put(67,-3.5){$(\!(\aa{\op}\bb){\op}\aa){\op}(\aa{\op}\bb)$}
\put(72,10){$(\aa{\op}\bb){\op}\aa$}
\psbezier[linewidth=2pt,border=3pt,linestyle=dashed](82,0)(95,0)(95,20)(82,20)
\psbezier[linewidth=2pt,border=3pt,linestyle=dashed](82,8)(87,8)(87,16)(82,16)
\psline[linewidth=2pt,border=3pt,linestyle=dashed](25,20)(82,20)
\psline[linewidth=2pt,border=3pt,linestyle=dashed](82,16)(25,16)
\psbezier[linewidth=2pt,border=3pt,linestyle=dashed](25,0)(12,0)(12,20)(25,20)
\psbezier[linewidth=2pt,border=3pt,linestyle=dashed](25,8)(20,8)(20,16)(25,16)
\psline[linewidth=2pt,border=3pt]{->}(25,0)(27,0)
\psline[linewidth=2pt,border=3pt]{->}(25,8)(27,8)
\end{picture}
\caption{\sf\smaller A link diagram for the trefoil knot of Figure~\ref{F:Diagrams} that is the closure (dashed lines) of a braid diagram, here $\sigg13$; the fundamental quandle of the knot is the quandle whose presentation is obtained by equating the labels on the left- and right-ends, here $\langle\aa, \bb \mid ((\aa{\op}\bb){\op}\aa){\op}(\aa{\op}\bb) = \aa, (\aa{\op}\bb){\op}\aa = \bb\rangle$, which is also $\langle\aa, \bb \mid \bb{\op}(\aa{\op}\bb) = \aa, (\aa{\op}\bb){\op}\aa = \bb\rangle$, or, more symmetrically, $\langle\aa, \bb, \cc \mid \aa{\op}\bb= \cc, \bb{\op}\cc = \aa, \cc{\op}\aa = \bb\rangle$. By projecting the fundamental quandle to a group, that is, interpreting~$\op$ as a conjugacy operation, one obtains the Wirtinger presentation of the fundamental group of the complement of the knot, here the group $\langle \aa, \bb \mid \aa\bb\aa = \bb\aa\bb\rangle$ in which we recognize~$\BR3$.}
\label{F:FundQuandle}
\end{figure}

The second approach consists, as in the case of braids, in fixing one auxiliary quandle (the same for all diagrams) and using it to define topological invariants. Here applications are so numerous that we can only be extremely sketchy and refer for instance to the survey~\cite{CarSurv} for a better account and a more complete bibliography. Typically, if $\SS$ is a finite quandle, one can count how many $\SS$-colourings exist. More precisely, the value of the \emph{quandle counting} invariant for a link~$\LL$ is defined to be the number of homomorphisms from the fundamental quandle~$\QQ_\LL$ to~$\SS$. It is shown in~\cite{Ino, HaNe} that the counting invariants associated with certain explicit family of quandles lead to classical link invariants like the linking number or the Alexander polynomial.

\subsubsection*{The case of Laver tables}

Laver tables are very far from all racks and quandles that have been mentioned above---and, much more generally, from those that have been used so far. By the way, Laver tables are LD-systems, but they are \emph{not} racks nor \emph{a fortiori} quandles: as asserted in Theorem~\ref{T:Tables}\ITEM3, the period of~$2^\nn -1$ in~$\AA_\nn$ is~$1$, meaning that the row of~$2^\nn-1$ is constant (with value~$2^\nn$) and, for $\nn \ge 1$, the associated left-translations is very far from bijective. So, the only direct application is the existence of an Hurwitz action for positive braids:

\begin{lemm}
\label{L:ActionBis}
For every~$\nn$, putting
\begin{equation} 
\label{E:LaverAction}
(\aa_1 \wdots \aa_\mm) \act \sig\ii = (\aa_1 \wdots \aa_{\ii-1}, \aa_\ii \op \aa_{\ii+1}, \aa_\ii, \aa_{\ii+2} \wdots \aa_\mm)
\end{equation}
defines an action of the monoid~$\BP\mm$ on~$\AA_\nn^\mm$.
\end{lemm}

The problem now is the failure of the laws \eqref{E:RI} and~\eqref{E:RII}, which respectively correspond to Reidemeister moves of types~I and~II. As for Reidemeister moves of type~I, we know that the problem vanishes if we restrict to braids; in the case of links, forgetting type~I amounts to restricting to what is called \emph{regular isotopy}, corresponding to considering framed links in which, in addition to the strands, a distinguished orthogonal direction is fixed at each point. The overall conclusion is that the failure of~\eqref{E:RI} alone does not discard topological applications. By the way, a number of recent works consist in extending to general racks some results first established in the particular case of quandles, see for instance~\cite{Nel, ChNe, NeWi}.

The failure of~\eqref{E:RII} is a more serious obstruction since it \emph{a priori} discards the existence of an Hurwitz action for arbitrary braids. However, it turns out that, at least in good cases, the problem can be solved. So assume that $(\SS, \op)$ is an LD-system. We do not assume that $(\SS, \op)$ is a rack, but we assume for a while that $(\SS, \op)$ is left-cancellative, that is, the left-translations of~$\op$ are injective. Then using for the negative crossings the colouring rule
\begin{equation*}
\VR(6,5)\begin{picture}(80,0)(0,4)
\psbezier[linewidth=2pt,border=3pt](0,8)(4,8)(4,0)(8,0)
\psbezier[linewidth=2pt,border=3pt](0,0)(4,0)(4,8)(8,8)
\pcline[linewidth=2pt,border=3pt](-1,0)(0,0)
\pcline[linewidth=2pt,border=3pt](-1,8)(0,8)
\pcline[linewidth=2pt,border=3pt]{->}(8,8)(10,8)
\pcline[linewidth=2pt,border=3pt]{->}(8,0)(10,0)
\put(-4,7.5){$\bb$}
\put(-4,-0.5){$\aa$}
\put(11.5,7.5){the unique $\cc$ satisfying $\aa \op \cc = \bb$, if such one exists}
\put(11.5,-0.5){$\bb$}
\end{picture}
\end{equation*}
enables one to define a \emph{partial} Hurwitz action, in the sense that $\vec\aa \act \ww$ need not be defined for every sequence~$\vec\aa$ in~$\SS^\mm$ and every $\mm$-strand braid word~$\ww$: not all sequences of initial colours can be propagated throughout the braid diagram. Then the point is the following (absolutely nontrivial) result:  

\begin{lemm}\cite{Dfb} 
\label{L:PartialAction}
Assume that $(\SS, \op)$ is a left-cancellative LD-system.

\ITEM1 For all $\mm$-strand braid words $\ww_1 \wdots \ww_\pp$, there exists at least one sequence~$\vec\aa$ in~$\SS^\mm$ such that $\vec\aa \act \ww_\ii$ is defined for every~$\ii$.

\ITEM2 If $\ww, \ww'$ are equivalent $\mm$-strand braid words, and $\vec\aa$ is a sequence in~$\SS^\mm$ such that both $\vec\aa \act \ww$ and $\vec\aa \act \ww'$
are defined, then the latter sequences are equal.
\end{lemm}

In other words, although $(\SS, \op)$ is not assumed to be a rack, one obtains an action that is partial but still enjoys good invariance properties. Applying this approach in the case when $(\SS, \op)$ is a free LD-system directly led to the orderability of the group~$\BR\mm$ in~\cite{Dfb}: free LD-systems are orderable, in the sense that there exists a linear ordering satisfying $\aa < \aa\op\bb$ for all~$\aa, \bb$, and using the associated colourings naturally leads to ordering braids: a braid~$\br$ is declared smaller than another braid~$\br'$ if, for some/any sequence~$\vec\aa$ such that both $\vec\aa \act \br$ and $\vec\aa \act \br'$ are defined, the sequence $\vec\aa \act \br$ is smaller than the sequence $\vec\aa \act \br'$ with respect to the lexicographical ordering on~$\SS^\mm$.

Laver tables are not left-cancellative, hence they are not directly eligible for Lemma~\ref{L:PartialAction} and further tricks will have to be developed in order to use them for colourings. A natural but probably too naive approach could be to use fractionary decompositions of braids: every $\mm$-strand braid~$\br$ can be expressed as a quotient $\br_1\inv \br_2$ where $\br_1$ and~$\br_2$ are positive $\mm$-strand braids, and the decomposition is unique if one requires in addition that $\br_1$ and~$\br_2$ admit no common left-divisor in the monoid~$\BP\mm$. Whenever $(\SS, \op)$ is an LD-system, the sequences $\vec\aa \act \br_1$ and $\vec\aa \act \br_2$ are defined for every sequence~$\vec\aa$ in~$\SS^\mm$ and, therefore, the pair $(\vec\aa \act \br_1, \vec\aa \act \br_2)$, which depends only on~$\vec\aa$ and~$\br$, could be used as a (sort of) colouring for~$\br$, see Figure~\ref{F:ColLaver} for an example. Alternatively, every braid in~$\BR\mm \setminus \BP\mm$ admits a unique expression as $\Delta_\mm^{-\dd} \br_0$ where $\Delta_\mm$ is Garside's fundamental $\mm$-strand braid, $\dd$ is a positive integer and $\br_0$ is a positive braid that is not left-divisible by~$\Delta_\mm$ in~$\BP\mm$, and one could use the pair $(\vec\aa \act \Delta_\mm^\dd, \vec\aa \act \br_0)$ as another colouring for~$\br$.

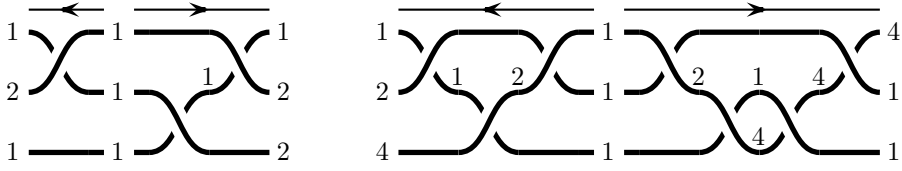
\begin{figure}[htb]
\begin{picture}(48,22)(0,1)
\psbezier[linewidth=2pt,border=3pt](0,16)(4,16)(4,8)(8,8)
\psbezier[linewidth=2pt,border=3pt](0,8)(4,8)(4,16)(8,16)
\psline[linewidth=2pt,border=3pt](0,0)(8,0)
\psline[linewidth=2pt,border=3pt](8,0)(10,0)
\psline[linewidth=2pt,border=3pt](8,8)(10,8)
\psline[linewidth=2pt,border=3pt](8,16)(10,16)
\psline[linewidth=2pt,border=3pt](14,0)(16,0)
\psline[linewidth=2pt,border=3pt](14,8)(16,8)
\psline[linewidth=2pt,border=3pt](14,16)(16,16)
\psbezier[linewidth=2pt,border=3pt](16,0)(20,0)(20,8)(24,8)
\psbezier[linewidth=2pt,border=3pt](16,8)(20,8)(20,0)(24,0)
\psline[linewidth=2pt,border=3pt](16,16)(24,16)
\psbezier[linewidth=2pt,border=3pt](24,8)(28,8)(28,16)(32,16)
\psbezier[linewidth=2pt,border=3pt](24,16)(28,16)(28,8)(32,8)
\psline[linewidth=2pt,border=3pt](24,0)(32,0)
\put(-3,-1){$1$}
\put(-3,7){$2$}
\put(-3,15){$1$}
\put(11,-1){$1$}
\put(11,7){$1$}
\put(11,15){$1$}
\put(33,-1){$2$}
\put(33,7){$2$}
\put(33,15){$1$}
\put(23,9){$1$}
\psline[linewidth=2pt]{->}(23.5,19)(24,19)
\psline[linewidth=2pt]{->}(4.5,19)(4,19)
\psline(0,19)(10,19)
\psline(14,19)(32,19)
\end{picture}
\begin{picture}(65,22)(0,1)
\psbezier[linewidth=2pt,border=3pt](0,16)(4,16)(4,8)(8,8)
\psbezier[linewidth=2pt,border=3pt](0,8)(4,8)(4,16)(8,16)
\psline[linewidth=2pt,border=3pt](0,0)(8,0)
\psbezier[linewidth=2pt,border=3pt](8,8)(12,8)(12,0)(16,0)
\psbezier[linewidth=2pt,border=3pt](8,0)(12,0)(12,8)(16,8)
\psline[linewidth=2pt,border=3pt](8,16)(16,16)
\psbezier[linewidth=2pt,border=3pt](16,16)(20,16)(20,8)(24,8)
\psbezier[linewidth=2pt,border=3pt](16,8)(20,8)(20,16)(24,16)
\psline[linewidth=2pt,border=3pt](16,0)(24,0)
\psline[linewidth=2pt,border=3pt](24,0)(26,0)
\psline[linewidth=2pt,border=3pt](24,8)(26,8)
\psline[linewidth=2pt,border=3pt](24,16)(26,16)
\psline[linewidth=2pt,border=3pt](30,0)(32,0)
\psline[linewidth=2pt,border=3pt](30,8)(32,8)
\psline[linewidth=2pt,border=3pt](30,16)(32,16)
\psbezier[linewidth=2pt,border=3pt](32,8)(36,8)(36,16)(40,16)
\psbezier[linewidth=2pt,border=3pt](32,16)(36,16)(36,8)(40,8)
\psline[linewidth=2pt,border=3pt](32,0)(40,0)
\psbezier[linewidth=2pt,border=3pt](40,0)(44,0)(44,8)(48,8)
\psbezier[linewidth=2pt,border=3pt](40,8)(44,8)(44,0)(48,0)
\psline[linewidth=2pt,border=3pt](40,16)(48,16)
\psbezier[linewidth=2pt,border=3pt](48,0)(52,0)(52,8)(56,8)
\psbezier[linewidth=2pt,border=3pt](48,8)(52,8)(52,0)(56,0)
\psline[linewidth=2pt,border=3pt](48,16)(56,16)
\psbezier[linewidth=2pt,border=3pt](56,8)(60,8)(60,16)(64,16)
\psbezier[linewidth=2pt,border=3pt](56,16)(60,16)(60,8)(64,8)
\psline[linewidth=2pt,border=3pt](56,0)(64,0)
\put(-3,-1){$4$}
\put(-3,7){$2$}
\put(-3,15){$1$}
\put(27,-1){$1$}
\put(27,7){$1$}
\put(27,15){$1$}
\put(65,-1){$1$}
\put(65,7){$1$}
\put(65,15){$4$}
\put(7,9){$1$}
\put(15,9){$2$}
\put(39,9){$2$}
\put(47,1){$4$}
\put(47,9){$1$}
\put(55,9){$4$}
\psline[linewidth=2pt]{->}(48.5,19)(49,19)
\psline[linewidth=2pt]{->}(11.5,19)(11,19)
\psline(30,19)(64,19)
\psline(0,19)(26,19)
\end{picture}
\caption{\sf\smaller Two tentative colourings of the $3$-strand braid $\sig1\sig2\siginv1$ using the Laver table~$\AA_2$: on the left, we use the decomposition of the braid as an irreducible fraction, namely $\siginv2 \sig1\sig2$, on the right, we use its decomposition with a denominator that is a power of Garside's fundamental braid, namely $\Delta_3\inv \sig2\sigg12\sig2$; in both cases, one propagates the colours from the middle.}
\label{F:ColLaver}
\end{figure}

The failure of left-cancellativity for each of the LD-systems~$\AA_\nn$ implies that we may have $\vec\aa \act \br_1 = \vec\bb \act \br_1$ with $\vec\aa \not= \vec\bb$ and, from there, with $\vec\aa \act \br_2 \not= \vec\bb \act \br_2$. However, an important positive point is that, by Laver's Theorem~\ref{T:Tables} and at least if Axiom~$\AxI3$ is true, (a subsystem of) the inverse limit~$\AA_\infty$ of the LD-systems~$\AA_\nn$ is a free LD-system. So, $\AA_\nn$-colourings can be viewed as finite approximations of free LD-system-colourings. Hence the left-cancellativity of free LD-systems might imply a good asymptotic behaviour for $\AA_\nn$-colourings. This is probably worth exploring.

Another (related) direction of research would be to use the approach of R.L.\,Rub\-insztein in~\cite{Rub}, based on the introduction of a notion of topological quandle. Laver tables are finite, discrete structures, and they are \textit{a priori} not relevant for such a topological approach. However, the limit~$\AA_\infty$ of the inverse system $(\AA_\nn, \proj_\nn)$ consists of all $2$-adic integers, and this limit is therefore equipped with a natural valuation, hence with an ultrametric topology. In particular, at least if Axiom~$\AxI3$ is satisfied, the substructure of~$\AA_\infty$ generated by~$(1, 1, ...)$ is a free LD-system: the latter is not a rack, but it is close to be one in that all left-translations are one-to-one. Thus investigating the counterpart of the space of colourings~$J_Q(L)$, a link invariant defined in~\cite{Rub}, seems to be a natural and promising approach.

\subsection{The (co)-homology approach}
\label{SS:Homology}

Owing to the difficulty of computing the fundamental quandle of a link, it is natural to try to obtain partial information by developing a convenient homology theory, viewed as a way to define sort of linear approximations. Initiated by R.\,Fenn, C.\,Rourke, B.\,Sanderson from 1990~\cite{FennTackling, FRSJames} and developed by S.\,Carter, M.\,Elhamdadi, M.\,Saito, and their collaborators in~\cite{CJKS, CES, CES2, CKS}, this approach proved to be extremely fruitful, as explained in~\cite{FennTackling, CarSurv}. 

\subsubsection*{The general principle}

A comprehensive survey can be found in~\cite{CarSurv}, and we shall just present here the very first steps. As we consider left-selfdistributivity here, it is coherent to use a symmetric version of the construction as developed in~\cite{CJKS} or (for the case of a general LD-system)~\cite{PrzHomo}. The starting observation is that several ways of associating chain complexes to an LD-system exist---and even more exist when one starts with a \emph{multi-LD-system}, that is, a set equipped with several mutually distributive operations~\cite{PrzPut}. 

\begin{lemm}
\label{L:Bicomplex}
Assume that $(\SS, \op)$ is an LD-system. For~$\kk \ge 1$, let $\CC_\kk(\SS)$ be a free $\ZZZZ$-module based on~$\SS^{\kk}$, and put $\CC_0(\SS)=\ZZZZ$. For $1 \le i \le \kk$, define $\ZZZZ$-linear maps $\dd^{\,\op}_{\kk, \ii}, \dd^{\, 0}_{\kk, \ii} : \CC_\kk(\SS) \rightarrow \CC_{\kk-1}(\SS)$ by
\begin{gather*}
\dd^{\,\op}_{\kk, \ii}(x_1 \wdots  x_\kk) = (x_1 \wdots  x_{i-1}, \widehat{\xx_\ii}, x_i \op x_{i+1} \wdots  x_i \op x_{\kk}), \\
\dd^{\,0}_{\kk, \ii}(x_1 \wdots  x_\kk) = (x_1 \wdots  x_{i-1}, \widehat{x_i}, x_{i+1} \wdots x_{\kk}).
\end{gather*}
Put $\der^{\,\op}_{\kk}:= \sum_{i=1}^\kk (-1)^{i-1} \dd^{\,\op}_{\kk, \ii},$ and $\der^{\,0}_{\kk}:= \sum_{i=1}^\kk (-1)^{i-1} \dd^{\, 0}_{\kk, \ii}$. Then, for every $\ZZZZ$-linear combination $\der_{\kk}$ of~$\der^{\,\op}_{\kk}$ and $\der^{\,0}_{\kk}$, we have  $\der_{\kk-1} \circ \der_{\kk} = 0$ for every~$\kk$.
\end{lemm}

\begin{proof}[Proof (sketch)]
A direct computation shows that, for all $1 \le j < i \le \kk$ and for every choice of $\diamond$ and $\star$ in~$\{\op, 0\}$, the relation $\dd^{\,\diamond}_{\kk-1,j} \circ \dd^{\,\star}_{\kk,i}= \dd^{\,\star}_{\kk-1,i-1} \circ \dd^{\,\diamond}_{\kk,j}$ is satisfied. From there, one deduces 
\begin{equation}
\label{E:Bicomplex}
\der^{\,\op}_{\kk-1} \circ \der^{\,\op}_{\kk} = \der^{\,0}_{\kk-1} \circ \der^{\,0}_{\kk} = \der^{\,\op}_{\kk-1} \circ \der^{\,0}_{\kk}  + \der^{\,0}_{\kk-1} \circ \der^{\,\op}_{\kk}  = 0,
\end{equation}
and the result easily follows. The point in this computation is that, when say $\dd^*_{\kk-1} \dd^*_\kk(\xx_1 \wdots \xx_{\kk+1})$ is expanded into a sum of $(\kk-1)\kk$ terms, then, for all $\ii < \jj$, there appear exactly two terms in which $\xx_\ii$ and~$\xx_\jj$ do not appear on the right:
\begin{multline*}
\qquad(-1)^{\ii + \jj + 2} (\xx_1 \wdots \xx_{\ii-1}, \widehat{\xx_\ii}, \xx_\ii \op \xx_{\ii+1} \wdots \xx_\ii \op \xx_{\jj-1}, \widehat{\xx_\jj}, \\ 
\xx_\ii \op (\xx_\jj\op\xx_{\jj+1}) \wdots \xx_\ii \op (\xx_\jj\op\xx_{\kk+1})),\qquad
\end{multline*}
which corresponds to omitting the $\jj$th entry first and then the $\ii$th one, and
\begin{multline*}
\qquad(-1)^{\ii + \jj + 1} (\xx_1 \wdots \xx_{\ii-1}, \widehat{\xx_\ii}, \xx_\ii \op \xx_{\ii+1} \wdots \xx_\ii \op \xx_{\jj-1}, \widehat{\xx_\jj}, \\
(\xx_\ii \op \xx_\jj) \op (\xx_\ii \op\xx_{\jj+1}) \wdots \xx_\ii \op (\xx_\ii \op\xx_{\jj+1}) \wdots \xx_\ii \op \xx_{\kk+1})),\qquad
\end{multline*}
which corresponds to omitting the $\ii$th entry first and then the $\jj-1$st one. The left-selfdistributivity law is then exactly the condition needed to ensure that the above two tuples coincide, so their cumulated contribution vanishes. 
\end{proof}

So Lemma~\ref{L:Bicomplex} says that, for every linear combination~$\der_\kk$ of~$\der_\kk^{\op}$ and~$\der_\kk^0$, the sequence $(\CC_\kk(\SS), \der_\kk)_\kk$ is a chain complex---actually \eqref{E:Bicomplex} says that $(\CC_\kk(\SS), \der_\kk^{\op}, \der_\kk^0)_\kk$ is what is called a chain bicomplex---leading to a derived notion of homology and, dually, of cohomology. It is standard to consider two particular linear combinations, namely $\der_\kk^*$ itself, and $\der_\kk^* - \der_\kk^0$.

\begin{defi}
\label{D:DistrHom}
Assume that $(\SS, \op)$ is an LD-system.

\ITEM1 For~$\kk \ge 1$, put $\der^\Rack_{\kk} = \der^{\,\op}_{\kk}-\der^{\,0}_{\kk}$. Then the chain complex $(\CC_\kk(\SS),\der^\Rack_{\kk})_\kk$ is called the \emph{rack complex} of~$(\SS, \op)$, and its homology is called the \emph{rack homology} of~$(\SS, \op)$, denoted by~$(\HH^\Rack_\kk(\SS))_\kk$.

\ITEM2 For every abelian group~$\GG$, define $\CC^\kk(\SS; \GG)$ to be $\Hom_{\ZZZZ}(\CC_\kk(\SS); \GG)$ and let $\der_\Rack^{\kk}$ be the differential on~$\CC^\kk(\SS; \GG)$ induced by $\der^\Rack_{\kk}$. The cohomology of the cochain complex $(\CC^\kk(\SS; \GG), \der_\Rack^{\kk})_\kk$ is called the $\GG$-valued \emph{rack cohomology} of~$(\SS, \op)$, denoted by~$(\HH_\Rack^\kk(\SS; \GG))_\kk$. The image of~$\der_\Rack^{\kk-1}$ (\resp the kernel of~$\der_\Rack^\kk$) is denoted by~$\BB_\Rack^\kk(\SS; \GG)$ (\resp $\ZZ_\Rack^\kk(\SS; \GG)$) and its elements are called $\GG$-valued \emph{$\kk$-coboundaries} (\resp\emph{$\kk$-cocycles}).
 
\ITEM3 The \emph{one-term distributive homology} and \emph{cohomology} of~$(\SS, \op)$ are obtained by replacing~$\der^\Rack_{\kk}$ with~$\der^{\,\op}_{\kk}$ everywhere.
\end{defi}

In the distributive world, the one-term distributive complex can be seen as the analogue of the bar complex for associative algebras, whereas the rack complex is an analogue of the Hochschild complex. This was pointed out in \cite{PrzHomo} and explained in the context of a unifying braided homology theory in \cite{Lebed1}. 

It turns out that, in view of topological applications, the rack (co)homology is more suitable than the one-term distributive (co)homology. More specifically, the rack $2$-cocycles directly lead to interesting invariants. It follows from the explicit definitions of Lemma~\ref{L:Bicomplex} that a map $\phi: \SS \times \SS \to \GG$ defines a (rack) $2$-cocycle if and only if it obeys the rule
\begin{equation}
\label{E:Cocycle}
\phi(\xx, \zz) + \phi(\xx\op\yy, \xx\op\zz) = \phi(\yy, \zz) + \phi(\xx, \yy\op\zz).
\end{equation}
Then the general principle that makes $2$-cocycles valuable here is the possibility of using them in the context of diagram colourings so as to obtain invariants.

\begin{lemm}\cite{CarSurv}
\label{L:Cocycle}
\ITEM1 Assume that $(\SS, \op)$ is an LD-system, $\GG$ is an abelian group, and $\phi: \SS \times \SS \to \GG$ is a $\GG$-valued $2$-cocycle for~$(\SS, \op)$. For $\DD$ a positive $\mm$-strand braid diagram and $\vec\aa$ in~$\SS^\mm$, define $\phih_\DD(\vec\aa) = \sum_\ii \phi(\aa_\ii, \bb_\ii)$ where $\aa_\ii, \bb_\ii$ are the input colours at the $\ii$th crossing of~$\DD$ when $\DD$ is coloured from~$\vec\aa$. Then $\phih_\DD$ is invariant under Reidemeister moves of type~III.

\ITEM2 If $(\SS, \op)$ is a rack and a negative crossing contributes $-\phi(\aa, \bb)$ when the output colours are~$\aa, \bb$, then $\phih_\DD$ is defined for every braid diagram and it is invariant under Reidemeister moves of type~II and~III.

\ITEM3 If $(\SS, \op)$ is a quandle and $\phi$ satisfies the rule $\phi(\xx, \xx) = 0$, then $\phih_\DD$ is defined for every link diagram and it is invariant under Reidemeister moves of type~I--III.
\end{lemm}

\begin{proof}
The argument for~\ITEM1 is given in Figure~\ref{F:Cocycle}. For~\ITEM2, concatenating two opposite crossings leads to a contribution of the form $\phi(\aa, \bb) - \phi(\aa, \bb)$. Finally, for~\ITEM3, adding a loop results in an additional contribution of the form $\pm\phi(\aa, \aa)$, hence $0$ under the additional assumption.
\end{proof}
 
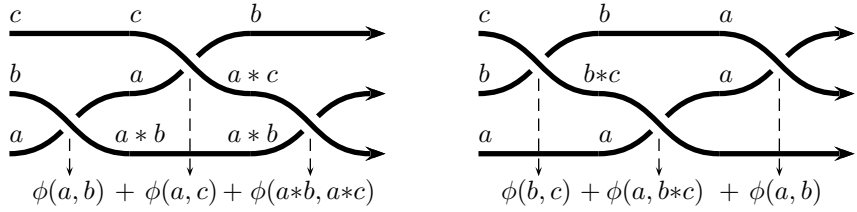
\begin{figure}[htb]
\begin{picture}(50,25)(0,-6)
\psbezier[linewidth=2pt,border=3pt](0,0)(8,0)(8,8)(16,8)
\psbezier[linewidth=2pt,border=3pt](0,8)(8,8)(8,0)(16,0)
\psline[linewidth=2pt,border=3pt](0,16)(16,16)
\psbezier[linewidth=2pt,border=3pt](16,8)(24,8)(24,16)(32,16)
\psbezier[linewidth=2pt,border=3pt](16,16)(24,16)(24,8)(32,8)
\psline[linewidth=2pt,border=3pt](16,0)(32,0)
\psbezier[linewidth=2pt,border=3pt](32,0)(40,0)(40,8)(48,8)
\psbezier[linewidth=2pt,border=3pt](32,8)(40,8)(40,0)(48,0)
\psline[linewidth=2pt,border=3pt](32,16)(48,16)
\psline[linewidth=2pt,border=3pt]{->}(48,0)(50,0)
\psline[linewidth=2pt,border=3pt]{->}(48,8)(50,8)
\psline[linewidth=2pt,border=3pt]{->}(48,16)(50,16)
\put(0,1.5){$\aa$}
\put(0,9.5){$\bb$}
\put(0,17.5){$\cc$}
\put(16,17.5){$\cc$}
\put(14,1.5){$\aa\op\bb$}
\put(16,9.5){$\aa$}
\put(29,9.5){$\aa\op\cc$}
\put(29,1.5){$\aa\op\bb$}
\put(32,17.5){$\bb$}
\psline[linewidth=0.5pt,linestyle=dashed]{->}(8,2)(8,-3)
\put(3,-6){$\phi(\aa, \bb)$}
\put(14,-6){$+$}
\psline[linewidth=0.5pt,linestyle=dashed]{->}(24,10)(24,-3)
\put(18,-6){$\phi(\aa, \cc)\,+$}
\psline[linewidth=0.5pt,linestyle=dashed]{->}(40,2)(40,-3)
\put(32,-6){$\phi(\aa{\op}\bb, \aa{\op}\cc)$}
\end{picture}
\hspace{10mm}
\begin{picture}(50,25)(0,-6)
\psbezier[linewidth=2pt,border=3pt](0,8)(8,8)(8,16)(16,16)
\psbezier[linewidth=2pt,border=3pt](0,16)(8,16)(8,8)(16,8)
\psline[linewidth=2pt,border=3pt](0,0)(16,0)
\psbezier[linewidth=2pt,border=3pt](16,0)(24,0)(24,8)(32,8)
\psbezier[linewidth=2pt,border=3pt](16,8)(24,8)(24,0)(32,0)
\psline[linewidth=2pt,border=3pt](16,16)(32,16)
\psbezier[linewidth=2pt,border=3pt](32,8)(40,8)(40,16)(48,16)
\psbezier[linewidth=2pt,border=3pt](32,16)(40,16)(40,8)(48,8)
\psline[linewidth=2pt,border=3pt](32,0)(48,0)
\psline[linewidth=2pt,border=3pt]{->}(48,0)(50,0)
\psline[linewidth=2pt,border=3pt]{->}(48,8)(50,8)
\psline[linewidth=2pt,border=3pt]{->}(48,16)(50,16)
\put(0,1.5){$\aa$}
\put(0,9.5){$\bb$}
\put(0,17.5){$\cc$}
\put(16,17.5){$\bb$}
\put(16,1.5){$\aa$}
\put(14,9.5){$\bb{\op}\cc$}
\put(32,9.5){$\aa$}
\put(32,17.5){$\aa$}
\psline[linewidth=0.5pt,linestyle=dashed]{->}(8,10)(8,-3)
\put(3,-6){$\phi(\bb, \cc)$}
\put(13.5,-6){$+$}
\psline[linewidth=0.5pt,linestyle=dashed]{->}(24,2)(24,-3)
\put(17,-6){$\phi(\aa, \bb{\op}\cc)$}
\put(32,-6){$+$}
\psline[linewidth=0.5pt,linestyle=dashed]{->}(40,10)(40,-3)
\put(36,-6){$\phi(\aa, \bb)$}
\end{picture}
\caption{\sf\smaller Using a $2$-cocycle to construct a braid invariant: one associates with every braid diagram the sum of the values of the cocycle at the successive crossings labelled by means of the reference LD-system; the cocycle rule of~\eqref{E:Cocycle} is exactly what is needed to guarantee invariance with respect to Reidemeister moves of type~$\mathrm{III}$.}
\label{F:Cocycle}
\end{figure}

Rack $3$-cocycles also proved to lead to interesting topological applications, but here we shall only refer to the survey~\cite{CarSurv} where a complete discussion can be found. 

\subsubsection*{The case of the Laver tables}

The Laver tables are directly eligible for the above constructions, and there is no problem for defining the associated homology and cohomology groups. The cases of one-term and rack homologies are rather different, the latter turning out to be much richer than the former. 

So, let us first briefly consider the one-term homology of Laver tables. As is the case of many monogenerated LD-systems (that is, LD-systems generated by a single element), the groups~$\HH^{\op}_\kk(\AA_\nn)$ are trivial:

\begin{prop}
\label{P:LaverHomol}
For every~$\nn$, the chain complex $(\CC_\kk(\AA_\nn), \der^{\op}_\kk)_\kk$ is acyclic, and the resulting homology groups~$\HH^*_\kk(\AA_\nn)$ are trivial.
\end{prop}

\begin{proof}
We follow the method of~\cite[Proposition~6.5]{PrzHomo} and give two different arguments. First define $\theta_\kk : \CC_\kk(\AA_\nn) \to\nobreak \CC_{\kk+1}(\AA_\nn)$ for $\kk \ge -1$ by $\theta_{-1}(1) = -(2^\nn)$ and $\theta_\kk(\xx_1 \wdots \xx_{\kk+1}) = -(2^\nn, \xx_1 \wdots \xx_\kk)$. Using the fact that $2^\nn \op \xx = \xx$ holds for every~$\xx$ in~$\AA_\nn$, we obtain 
\begin{align*}
\VR(0,4)\theta_\kk\der^*_\kk(\xx_1 \wdots \xx_{\kk+1})
&= \smash{\sum_{\ii=1}^{\kk+1}} (-1)^\ii (2^\nn, \xx_1 \wdots \xx_{\xx_1}, \widehat{\xx_\ii}, \xx_\ii{\op}\xx_{\ii+1} \wdots \xx_\ii{\op}\xx_{\kk+1}),\\
\theta_{\kk+1}\der^*_{\kk+1}(\xx_1 \wdots \xx_{\kk+1})
&= (\xx_1 \wdots \xx_\kk)\\
\VR(5,3)&\hspace{0mm}+ \smash{\sum_{\ii=1}^{\kk+1}} (-1)^{\ii+1}(2^\nn, \xx_1 \wdots \xx_{\xx_1}, \widehat{\xx_\ii}, \xx_\ii{\op}\xx_{\ii+1} \wdots \xx_\ii{\op}\xx_{\kk+1}),
\end{align*}
whence $\theta_\kk\der^*_\kk + \der^*_{\kk+1}\theta_{\kk+1} = \id$. Hence $\theta_\kk$ is a contracting homotopy for~$(\CC_\kk(\AA_\nn), \der^*_\kk)$, and the homology of the complex must be trivial.

Putting $\theta'_{-1}(1) = (2^\nn)$ and $\theta'_\kk(\xx_1 \wdots \xx_{\kk+1}) = (-1)^{\kk+1}(\xx_1 \wdots \xx_\kk, 2^\nn)$, one checks that $\theta'_*$ is an alternative contracting homotopy for~$(\CC_\kk(\AA_\nn), \der^*_\kk)$ now using the fact that $\xx \op 2^\nn = 2^\nn$ holds for every~$\xx$ in~$\AA_\nn$.
\end{proof}

Rack (co)homology of Laver tables is much more interesting. Due to our specific interest in topological applications, and owing to Lemma~\ref{L:Cocycle}, we only consider rack $2$-cocycles. Without loss of generality, we also restrict to $\ZZZZ$-valued cocycles. Thus, we are interested in maps $\phi: \{1 \wdots 2^\nn\} \times \{1 \wdots 2^\nn\} \to \ZZZZ$ that obey~\eqref{E:Cocycle}. It turns out that such $2$-cocycles can be described very precisely in terms of the values that appear in the columns of the tables~$\AA_\nn$. Here we shall mention the main result only, and refer to~\cite{Dik} for more details and proofs. 

\begin{prop}\cite{Dik}
\label{P:Basis}
For every~$\nn$, the $\ZZZZ$-valued $2$-cocycles for~$\AA_\nn$ make a free $\ZZZZ$-module of rank~$2^\nn$, with a basis consisting of coboundaries defined for $1 \le \qq < 2^\nn$~by 
$$\psi_{\qq, \nn}(\xx, \yy) = \begin{cases}
\ 1&\mbox{if $\qq$ appears in the column of $\yy$ in~$\AA_\nn$ but not in that of $\xx \op \yy$},\\
\ 0&\mbox{otherwise},\end{cases}$$
completed with the constant cocycle with value~$1$.
\end{prop}

A complete enumeration in the case of the $8$-element table~$\AA_3$ is displayed in Table~\ref{T:Psi}.

\begin{table}[htb]
\newdimen\arrayitem
\def\arrayitem{2.4mm}
\smaller\smaller
\begin{tabular}{c|c}
\VR(0,1.5)\hbox to \arrayitem{\hspace{-2mm}$\psi_{1,3}$}&\1\2\3\4\5\6\7\8\\
\hline
\VR(3,0)$1$&\1\0\0\0\0\0\0\0\\
$2$&\1\0\0\0\0\0\0\0\\
$3$&\1\0\0\0\0\0\0\0\\
$4$&\1\0\0\0\0\0\0\0\\
$5$&\1\0\0\0\0\0\0\0\\
$6$&\1\0\0\0\0\0\0\0\\
$7$&\1\0\0\0\0\0\0\0\\
$8$&\0\0\0\0\0\0\0\0
\end{tabular}
\quad
\begin{tabular}{c|c}
\VR(0,1.5)\hbox to \arrayitem{\hspace{-2mm}$\psi_{2,3}$}&\1\2\3\4\5\6\7\8\\
\hline
\VR(3,0)$1$&\0\1\0\0\0\0\0\0\\
$2$&\1\1\0\0\1\0\0\0\\
$3$&\1\1\0\0\1\0\0\0\\
$4$&\0\1\0\0\0\0\0\0\\
$5$&\1\1\0\0\1\0\0\0\\
$6$&\1\1\0\0\1\0\0\0\\
$7$&\1\1\0\0\1\0\0\0\\
$8$&\0\0\0\0\0\0\0\0
\end{tabular}
\quad
\begin{tabular}{c|c}
\VR(0,1.5)\hbox to \arrayitem{\hspace{-2mm}$\psi_{3,3}$}&\1\2\3\4\5\6\7\8\\
\hline
\VR(3,0)$1$&\1\0\1\0\1\0\0\0\\
$2$&\0\0\1\0\0\0\0\0\\
$3$&\1\0\1\0\1\0\0\0\\
$4$&\0\0\1\0\0\0\0\0\\
$5$&\1\0\1\0\1\0\0\0\\
$6$&\1\0\1\0\1\0\0\0\\
$7$&\1\0\1\0\1\0\0\0\\
$8$&\0\0\0\0\0\0\0\0
\end{tabular}
\quad
\begin{tabular}{c|c}
\VR(0,1.5)\hbox to \arrayitem{\hspace{-2mm}$\psi_{4,3}$}&\1\2\3\4\5\6\7\8\\
\hline
\VR(3,0)$1$&\0\0\0\1\0\0\0\0\\
$2$&\0\0\0\1\0\0\0\0\\
$3$&\0\1\0\1\0\1\0\0\\
$4$&\0\0\0\1\0\0\0\0\\
$5$&\0\1\0\1\0\1\0\0\\
$6$&\0\1\0\1\0\1\0\0\\
$7$&\1\1\1\1\1\1\1\0\\
$8$&\0\0\0\0\0\0\0\0
\end{tabular}

\VR(1,0)

\begin{tabular}{c|c}
\VR(0,1.5)\hbox to \arrayitem{\hspace{-2mm}$\psi_{5,3}$}&\1\2\3\4\5\6\7\8\\
\hline
\VR(3,0)$1$&\1\0\0\0\1\0\0\0\\
$2$&\1\0\0\0\1\0\0\0\\
$3$&\1\0\0\0\1\0\0\0\\
$4$&\0\0\0\0\0\0\0\0\\
$5$&\1\0\0\0\1\0\0\0\\
$6$&\1\0\0\0\1\0\0\0\\
$7$&\1\0\0\0\1\0\0\0\\
$8$&\0\0\0\0\0\0\0\0
\end{tabular}
\quad
\begin{tabular}{c|c}
\VR(0,1.5)\hbox to \arrayitem{\hspace{-2mm}$\psi_{6,3}$}&\1\2\3\4\5\6\7\8\\
\hline
\VR(3,0)$1$&\0\1\0\0\0\1\0\0\\
$2$&\0\1\0\0\0\1\0\0\\
$3$&\1\1\1\0\1\1\1\0\\
$4$&\0\0\0\0\0\0\0\0\\
$5$&\0\1\0\0\0\1\0\0\\
$6$&\0\1\0\0\0\1\0\0\\
$7$&\1\1\1\0\1\1\1\0\\
$8$&\0\0\0\0\0\0\0\0
\end{tabular}
\quad
\begin{tabular}{c|c}
\VR(0,1.5)\hbox to \arrayitem{\hspace{-2mm}$\psi_{7,3}$}&\1\2\3\4\5\6\7\8\\
\hline
\VR(3,0)$1$&\1\0\1\0\1\0\1\0\\
$2$&\0\0\0\0\0\0\0\0\\
$3$&\1\0\1\0\1\0\1\0\\
$4$&\0\0\0\0\0\0\0\0\\
$5$&\1\0\1\0\1\0\1\0\\
$6$&\0\0\0\0\0\0\0\0\\
$7$&\1\0\1\0\1\0\1\0\\
$8$&\0\0\0\0\0\0\0\0
\end{tabular}\\
\VR(1,0)
\caption{\sf\smaller A basis of $\BB_\Rack^2(\AA_3)$ consisting of the seven $\{0, 1\}$-valued $2$-cocycles~$\psi_{\qq, 3}$ with $1 \le \qq \le 7$. To make reading easier, the zeroes are indicated with ``-''. Completing with the constant cocycle with value~$1$, we obtain a basis of~$\ZZ_\Rack^2(\AA_3)$.}\label{T:Psi}
\end{table}

The proof of Proposition~\ref{P:Basis} is not trivial, and it relies on the combinatorial properties of right-division in Laver tables. Two-cocycles capture a lot of information about Laver tables: for instance, one can directly recover from the cocycle~$\psi_{2^{\nn-1}, \nn}$ all periods in~$\AA_\nn$, hence, in a sense, the most critical combinatorial parameters. 

Rack $3$-cocycles can also be analyzed for Laver tables. They involve functions of three variables and the $2$-cocycle condition~\eqref{E:Cocycle} is replaced with the $3$-cocycle condition
\begin{align}
\label{E:Cocycle3}
\phi(\xx \op \yy, \xx \op \zz, \xx \op \tt) + &\phi(\xx, \yy, \zz \op \tt) + \phi(\xx, \zz, \tt)\\ 
\notag
&= \phi(\xx, \yy \op \zz, \yy \op \tt) + \phi(\yy, \zz, \tt) + \phi(\xx, \yy, \tt).
\end{align}
It turns out that $3$-cocycles on~$\AA_\nn$ make a free $\ZZZZ$-module of rank $2^{2\nn} - 2^\nn +1$, and an explicit basis can again be described~\cite{Dik}.

At the moment, the question of using the above results to extract topological information about not necessarily positive braids and possibly links, remains open, as does the question of a topologically interpreting these possible invariants. However, we note that having an explicit basis of  $2$-cocycles made of $\NNNN$-valued functions seems especially promising in view of combinatorial interpretations, typically for counting arguments. We shall not go further here, but, clearly, the conclusion of this section should be that the (co)homological approach is promising in terms of possible topological applications for Laver tables.

\subsection{The approach of the Yang--Baxter equation}
\label{SS:YBE}

Another context in which selfdistributive structures are involved is that of the (Quantum) Yang--Baxter equation (YBE or QYBE) and its connections with quantum groups and $R$-matrices, whence indirectly with topology and knot invariants. 

\subsubsection*{The general principle}

We start from the (non-parametric form of) the (quantum) Yang--Baxter equation, or, rather, of the equivalent braid equation.

\begin{defi}
\label{D:Solution}
If $V$ is a vector space, an element~$R$ of~$\GL(V \otimes V)$ is called a \emph{solution of the Yang--Baxter equation} (YBE), or an \emph{$R$-matrix}, if we have
\begin{equation}
\label{E:YBYB}
(\RR \otimes \id)(\id \otimes \RR)(\RR \otimes \id) = (\id \otimes \RR)(\RR \otimes \id)(\id \otimes \RR).
\end{equation}
\end{defi}

If one writes~$\RR^{\ii\jj}$ for the automorphism of~$V^{\otimes3}$ that corresponds to~$\RR$ acting on the $\ii$th and $\jj$th coordinates, the YBE becomes
\begin{equation}
\label{E:YBYB2}
\RR^{12} \RR^{23} \RR^{12} = \RR^{23} \RR^{12} \RR^{23},
\end{equation}
directly reminiscent of the braid relation~\eqref{E:BraidPres}---with the notation of \eqref{E:YBYB2}, the original Yang--Baxter equation is $\RR^{12} \RR^{13} \RR^{23} = \RR^{23} \RR^{13} \RR^{12}$; it transforms into the ``braid form''~\eqref{E:YBYB2} when $\RR$ is replaced by~$\Pi\RR$, where $\Pi$ is the switch operator that exchanges~$\xx$ and~$\yy$ \cite{Jim}.

For instance, if $A$ is $\CCCC[\qq, \qq\inv]$ and $\VV$ is $A \times A$ with standard basis~$(\ee_1, \ee_2)$, then the automorphism of $\VV \otimes \VV$ defined in the basis $(\ee_1 \otimes \ee_1$, $\ee_1 \otimes \ee_2$, $\ee_2 \otimes \ee_1$, $\ee_2 \otimes \ee_2)$ by the matrix $q^{-1/2}\left(\begin{smallmatrix}1&0&0&0\\0&0&1&0\\0&1&q-q\inv&0\\0&0&0&1\end{smallmatrix}\right)$ satisfies~\eqref{E:YBYB}, that is, it is a solution of YBE. This solution is connected with the basic representation of the quantum group~$U_\qq(\mathfrak{sl}(2))$ and the Jones polynomial~\cite{Kas}. 

Among the (many) solutions of the Yang--Baxter equation, we consider here those solutions~$\RR$ that preserve some fixed basis~$\SS$ of the considered vector space~$\VV$. Then the restriction of~$\RR$ to $\SS \times \SS$ yields a bijection~$\sol$ of~$\SS \times \SS$ to itself that satisfies 
\begin{equation}
\label{E:SetTheor}
\sol^{12} \sol^{23} \sol^{12} = \sol^{23} \sol^{12} \sol^{23},
\end{equation}
and, conversely, every bijection of~$\SS \times \SS$ into itself that satisfies \eqref{E:SetTheor} induces a solution of~YBE that maps~$\SS^{\otimes 2}$ into itself. Such solutions of YBE are called \emph{set-theoretic} because they are entirely determined by their action on the basis.

\begin{defi}
\label{D:SetTheor}
A \emph{set-theoretic solution of YBE} is a pair~$(\SS, \sol)$ where $\SS$ is a set and $\sol$ is a bijection of~$\SS \times \SS$ into itself that satisfies~\eqref{E:SetTheor}. In this case, we denote by $\sol_1(\xx, \yy)$ and $\sol_2(\xx, \yy)$ the first and the second entry of~$\sol(\xx, \yy)$. A set-theoretic solution $(\SS,\sol)$ of YBE is called \emph{nondegenerate} if, for every~$\aa$ in~$\SS$, the left-translation $\yy\mapsto \sol_1(\aa,\yy)$ is one-to-one and the right-translation $\xx\mapsto \sol_2(\xx,\aa)$ are one-to-one. 
\end{defi} 

A set-theoretic solution~$\sol$ of~YBE can then be characterized in terms of algebraic laws obeyed by the associated maps~$\sol_1$ and~$\sol_2$ viewed as binary operation on the reference set~$\SS$. 

\begin{lemm}
\label{L:SolRack}
Assume that $(\SS, \sol)$ is a set-theoretic solution of YBE. For $\aa, \bb$ in~$\SS$, write $\aa \opl \bb$ for~$\sol_1(\aa, \bb)$ and $\aa \opr \bb$ for $\sol_2(\aa, \bb)$. Then the operations~$\opl$ and~$\opr$ obey the laws
\begin{gather}
\label{E:Birack1}
(\xx\opl\yy) \opl ((\xx\opr\yy) \opl \zz) = \xx \opl (\yy\opl\zz),\\
\label{E:Birack2}
(\xx\opl\yy) \opr ((\xx\opr\yy) \opl \zz) = (\xx \opr (\yy\opl\zz)) \opl (\yy\opr\zz),\\
\label{E:Birack3}
(\xx\opr\yy) \opr \zz = (\xx \opr (\yy\opl\zz)) \opr (\yy\opr\zz).
\end{gather}
Conversely, if $\opl$ and $\opr$ are binary operations on~$\SS$ that satisfy \eqref{E:Birack1}--\eqref{E:Birack3} and $\sol$ is the map of $\SS\times\SS$ to itself defined by $\sol(\aa, \bb) = (\aa \opl \bb, \aa \opr \bb)$, then $(\SS, \sol)$ is a set-theoretic solution of YBE. In the above context, $(\SS, \sol)$ is nondegenerate if and only if left-translations of~$\opl$ and the right-translations of~$\opr$ are one-to-one.
\end{lemm}

\begin{proof}
We may appeal to braid colourings, using colours from~$\SS$ and the rule
\begin{equation}
\label{E:RuleBis}
\VR(6,6)\begin{picture}(12,0)(0,4)
\psbezier[linewidth=2pt,border=3pt](0,0)(4,0)(4,8)(8,8)
\psbezier[linewidth=2pt,border=3pt](0,8)(4,8)(4,0)(8,0)
\pcline[linewidth=2pt,border=3pt](-1,0)(0,0)
\pcline[linewidth=2pt,border=3pt](-1,8)(0,8)
\pcline[linewidth=2pt,border=3pt]{->}(8,8)(10,8)
\pcline[linewidth=2pt,border=3pt]{->}(8,0)(10,0)
\put(-4,7.5){$\bb$}
\put(-4,-0.5){$\aa$}
\put(11.5,7.5){$\aa \opr \bb$}
\put(11.5,-0.5){$\aa \opl \bb$,}
\end{picture}
\end{equation}
that is, the extension of~\eqref{E:Rules} in which both crossing strands may change colours. Then saying that $\sol$ satisfies~\eqref{E:SetTheor} amounts to saying that, for every choice of the input colours, the output colours of the diagrams $\sig1\sig2\sig1$ and $\sig2\sig1\sig2$ coincide. We read on Figure~\ref{F:Birack} that this happens exactly when the operation~$\opr$ and~$\opl$ obey the laws of~\eqref{E:Birack1}--\eqref{E:Birack3}. The other verifications are then straightforward.
\end{proof}

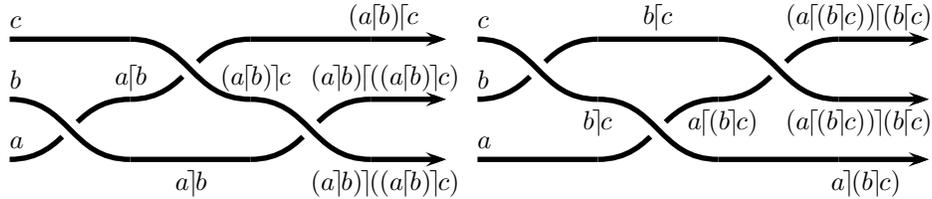
\begin{figure}[htb]
\begin{picture}(61,24)(0,-5)
\psbezier[linewidth=2pt,border=3pt](0,0)(8,0)(8,8)(16,8)
\psbezier[linewidth=2pt,border=3pt](0,8)(8,8)(8,0)(16,0)
\psline[linewidth=2pt,border=3pt](0,16)(16,16)
\psbezier[linewidth=2pt,border=3pt](16,8)(24,8)(24,16)(32,16)
\psbezier[linewidth=2pt,border=3pt](16,16)(24,16)(24,8)(32,8)
\psline[linewidth=2pt,border=3pt](16,0)(32,0)
\psbezier[linewidth=2pt,border=3pt](32,0)(40,0)(40,8)(48,8)
\psbezier[linewidth=2pt,border=3pt](32,8)(40,8)(40,0)(48,0)
\psline[linewidth=2pt,border=3pt](32,16)(48,16)
\psline[linewidth=2pt,border=3pt]{->}(48,0)(58,0)
\psline[linewidth=2pt,border=3pt]{->}(48,8)(58,8)
\psline[linewidth=2pt,border=3pt]{->}(48,16)(58,16)
\put(0,1.5){$\aa$}
\put(0,9.5){$\bb$}
\put(0,17.5){$\cc$}
\put(14,10){$\aa\!\opr\!\bb$}
\put(22,-4){$\aa\!\opl\!\bb$}
\put(28,10){$(\aa\!\opr\!\bb)\!\opl\!\cc$}
\put(45,18){$(\aa\!\opr\!\bb)\!\opr\!\cc$}
\put(40,-4){$(\aa\!\opl\!\bb)\!\opl\!((\aa\!\opr\!\bb)\!\opl\!\cc)$}
\put(40,10){$(\aa\!\opl\!\bb)\!\opr\!((\aa\!\opr\!\bb)\!\opl\!\cc)$}
\end{picture}
\begin{picture}(61,24)(0,-5)
\psbezier[linewidth=2pt,border=3pt](0,8)(8,8)(8,16)(16,16)
\psbezier[linewidth=2pt,border=3pt](0,16)(8,16)(8,8)(16,8)
\psline[linewidth=2pt,border=3pt](0,0)(16,0)
\psbezier[linewidth=2pt,border=3pt](16,0)(24,0)(24,8)(32,8)
\psbezier[linewidth=2pt,border=3pt](16,8)(24,8)(24,0)(32,0)
\psline[linewidth=2pt,border=3pt](16,16)(32,16)
\psbezier[linewidth=2pt,border=3pt](32,8)(40,8)(40,16)(48,16)
\psbezier[linewidth=2pt,border=3pt](32,16)(40,16)(40,8)(48,8)
\psline[linewidth=2pt,border=3pt](32,0)(48,0)
\psline[linewidth=2pt,border=3pt]{->}(48,0)(60,0)
\psline[linewidth=2pt,border=3pt]{->}(48,8)(60,8)
\psline[linewidth=2pt,border=3pt]{->}(48,16)(60,16)
\put(0,1.5){$\aa$}
\put(0,9.5){$\bb$}
\put(0,17.5){$\cc$}
\put(14,4){$\bb\!\opl\!\cc$}
\put(22,18){$\bb\!\opr\!\cc$}
\put(28,4){$\aa\!\opr\!(\bb\!\opl\!\cc)$}
\put(47,-4){$\aa\!\opl\!(\bb\!\opl\!\cc)$}
\put(41,18){$(\aa\!\opr\!(\bb\!\opl\!\cc))\!\opr\!(\bb\!\opr\!\cc)$}
\put(41,4){$(\aa\!\opr\!(\bb\!\opl\!\cc))\!\opl\!(\bb\!\opr\!\cc)$}
\end{picture}
\caption{\sf\smaller Colouring braids using the rule~\eqref{E:RuleBis} is invariant under braid relations if and only if the birack laws~\eqref{E:Birack1}--\eqref{E:Birack3} are obeyed.}
\label{F:Birack}
\end{figure}

The following terminology is then natural:

\begin{defi}
\label{D:Birack}
A \emph{birack} is a system~$(\SS, \opl, \opr)$ consisting of a set~$\SS$ equipped with two binary operations~$\opl$ and~$\opr$ that satisfy~\eqref{E:Birack1}--\eqref{E:Birack3} and such that the left-translations of~$\opl$ and the right-translations of~$\opr$ are one-to-one. 
\end{defi}

So Lemma~\ref{L:SolRack} says that a set-theoretic solution of the YBE, that is, a set-theoretic $R$-matrix, is one and the same thing as a birack. Let us mention here the beautiful result of Rump~\cite{Rum} who observed that inverting the operations of a birack enables one to replace biracks and the rather complicated laws~\eqref{E:Birack1}--\eqref{E:Birack3} with equivalent structures made of a set equipped with a binary operation obeying the unique more simple law $(\xx \op \yy) \op (\xx \op \zz) = (\yy \op \xx) \op (\yy \op \zz)$, see~\cite[Chapter~XII]{Garside}. 

Returning to selfdistributive structures, we immediately obtain the following simple connection:

\begin{lemm}
Assume that $\op$ is a binary operation on a set~$\SS$. For~$\aa, \bb$ in~$\SS$, define $\aa \op_0 \bb = \aa$. Then $(\SS, \op, \op_0)$ is a birack if and only if $(\SS, \op)$ is a rack.
\end{lemm}

The verification has already been done, as this essentially amounts to checking that what remains from \eqref{E:Birack1}--\eqref{E:Birack3} when the operation~$\opr$ is the trivial operation~$\op_0$ is the fact that the operation~$\op$ obeys the left-selfdistributivity law: this corresponds to specializing~\eqref{E:RuleBis} into~\eqref{E:Rules}, that is, using Figure~\ref{F:Translation} (case of type~III${}_{\scriptscriptstyle++}$) instead of Figure~\ref{F:Birack}. In terms of $R$-matrices, we deduce the following result, which appears in~\cite{Dri} and belongs to folklore:

\begin{prop}
\label{P:YBE}
Assume that $(\SS, \op)$ is a (finite) rack. Let $\VV$ be a $\CCCC$-vector space based on a copy $(\ee_\aa)_{\aa \in \SS}$ of~$\SS$. Then the endomorphism of~$\VV^{\otimes2}$ defined by $\RR(\ee_\aa, \ee_\bb) = (\ee_{\aa \op \bb}, \ee_\aa)$ is a (set-theoretic) solution of YBE.
\end{prop}

By definition, a solution of YBE is an automorphism of the considered space: in the context of Proposition~\ref{P:YBE}, the endomorphism~$\RR$ is invertible if and only if the map $(\aa, \bb) \mapsto (\aa \op \bb, \aa)$ is a bijection of~$\SS \times \SS$, that is, if the left-translations associated with~$\op$ are bijective. This is the place where the assumption that $(\SS, \op)$ is a rack, and not only a general LD-system, is used. 

\subsubsection{The case of the Laver tables}

When we consider the Laver tables, they are indeed finite LD-systems, but they are not racks (except in the trivial case of~$\AA_0$): in the table of~$\AA_\nn$, the row of~$2^\nn - 1$ is constant, hence very far from being bijective. The rest of the construction works, so one naturally obtains is a ``pseudo-solution'' of YBE, defined to be an endomorphism that satisfies~\eqref{E:YBYB2} but need not be invertible. For instance, the pseudo-$R$-matrix associated with the Laver table~$\AA_1$ corresponds to the (non-invertible) matrix $\left(\begin{smallmatrix}0&0&1&0\\0&0&0&0\\1&1&0&0\\0&0&0&1\end{smallmatrix}\right)$, whereas that associated with~$\AA_2$ is  
\VR(17,17)\smash{$\left(\begin{smallmatrix}
0&0&0&0&0&0&0&0&0&0&0&0&0&0&0&0\\
0&0&0&0&0&0&0&0&0&0&0&0&0&0&0&0\\
0&0&0&0&0&0&0&0&0&0&0&0&0&0&0&0\\
0&0&0&0&0&0&0&0&0&0&0&0&1&0&0&0\\
1&0&1&0&0&0&0&0&0&0&0&0&0&0&0&0\\
0&0&0&0&0&0&0&0&0&0&0&0&0&0&0&0\\
0&0&0&0&0&0&0&0&0&0&0&0&0&0&0&0\\
0&0&0&0&0&0&0&0&0&0&0&0&0&1&0&0\\
0&0&0&0&0&0&0&0&0&0&0&0&0&0&0&0\\
0&0&0&0&1&0&1&0&0&0&0&0&0&0&0&0\\
0&0&0&0&0&0&0&0&0&0&0&0&0&0&0&0\\
0&0&0&0&0&0&0&0&0&0&0&0&0&0&1&0\\
0&1&0&1&0&0&0&0&1&1&1&1&0&0&0&0\\
0&0&0&0&0&1&0&1&0&0&0&0&0&0&0&0\\
0&0&0&0&0&0&0&0&0&0&0&0&0&0&0&0\\
0&0&0&0&0&0&0&0&0&0&0&0&0&0&0&1
\end{smallmatrix}\right)$}. In general, the pseudo-$R$-matrix associated with the Laver table~$\AA_\nn$ is a square matrix of size~$2^{2\nn}$ that contains $2^\nn$ entries equal to~$1$. The obvious question is whether such ``non-invertible $R$-matrices'' can be of any use, typically in connection with the theory of Hopf algebras. In particular, one can wonder whether $\qq$-deformations of such matrices might exist and be useful. As in Subsection~\ref{SS:Colouring}, many results originally established using quandles have been subsequently extended to arbitrary racks~\cite{AFGV}. Also, racks proved to play a fundamental r\^ole in the classification of finite-dimensional pointed Hopf algebras~\cite{AnG}. The question of whether one could go one step further and work with more general LD-systems, specifically with Laver tables, remains open.

\section{The well-foundedness of the braid ordering}

The second result by Laver we shall mention here involves Artin's braid groups and their ordering(s). Braid groups were proved to be orderable, that is, to admit a left-invariant linear ordering, in 1992 \cite{Dfa, Dfb}, by an ordering that proved both to be canonical, in the sense that many different approaches converge to the same notion, and to have rich combinatorial properties~\cite{Dhr}. Using his approach to selfdistributivity via recursive normal forms, Rich Laver proved in 1995 what is probably the deepest result known so far about this braid ordering, namely that its restriction to the monoid of positive braids is a well-ordering. At the moment, this mainly led to applications of logical flavour, but the result inspires several promising ideas for further work.

This section contains four subsections. First, the braid ordering, Laver's result, and its direct consequences are described in Subsection~\ref{SS:Order}. Subsequent refinements are mentioned in Subsection~\ref{SS:Refinements}. Next, applications to unprovability statements are stated in Subsection~\ref{SS:Unprovability}. Finally, we discuss more hypothetic applications involving the Conjugacy Problem of braids and, possibly, the Markov equivalence relation, in Subsection~\ref{SS:Mu}. 

\subsection{The well-ordering of positive braids}
\label{SS:Order}

We recall from Subsection~\ref{SS:Colouring} that the $\nn$-strand braid group, that is, the group of isotopy classes of $\nn$-strand braid diagrams, is denoted by~$\BR\nn$. The group~$\BR\nn$ admits a more or less canonical family of generators (`the Artin generators') $\sig1 \wdots \sig{\nn-1}$ in terms of which $\BR\nn$ admits the presentation~\eqref{E:BraidPres}. Thus every $\nn$-strand braid is represented by various words in the alphabet $\{\sigg1{\pm1} \wdots \sigg{\nn-1}{\pm1}\}$, naturally called \emph{$\nn$-strand braid words}, two such braid words representing the same braid if and only if they can be transformed into one another using the relations of~\eqref{E:BraidPres} and the free group relations~$\sig\ii \siginv\ii = \siginv\ii \sig\ii = 1$.

\begin{defi}\cite{Dfb}
\label{D:BraidOrder}
\ITEM1 A braid word~$\ww$ is called \emph{$\sig\ii$-positive} if it contains the letter~$\sig\ii$ but neither the letter~$\siginv\ii$ nor any letter~$\sigg\jj{\pm1}$ with $\jj < \ii$. 

\ITEM2 For~$\br, \br'$ in~$B_\nn$, say that $\br \lD \br'$ holds if, among the various braid words that represent $\br\inv \br'$, at least one is $\sig\ii$-positive for some~$\ii$.
\end{defi}

In other words, $\br \lD \br'$ holds if the quotient-braid $\br\inv \br'$ admits an expression  in which the generator~$\sig\ii$ with least index occurs positively only. For instance, consider $\br = \sig1$ and $\br' = \sig2\sig1$. Then the quotient $\br\inv \br'$ is~$\siginv1 \sig2 \sig1$, so the braid word $\siginv1 \sig2 \sig1$ is one expression of this quotient, and it is neither $\sig1$-positive nor $\sig2$-positive. Now another expression of the same quotient-braid is $\sig2 \sig1 \siginv2$, which is a $\sig1$-positive braid word. Therefore $\sig1 \lD \sig2\sig1$ is declared to be true.

\begin{prop}\cite{Dfa, Dfb}
For every~$\nn$, the relation~$\lD$ is a linear ordering on the group~$\BR\nn$ and it is left-invariant, that is, $\br \lD \br'$ implies $\brr\br \lD \brr\br'$ for every~$\brr$.
\end{prop}

The braid order~$\lD$ will be referred to here as the \emph{D-ordering} of braids (`Dehornoy ordering'). There is no need to mention a braid index here, as one shows that the D-ordering on~$\BR{\nn-1}$ is the restriction of the D-ordering on~$\BR\nn$ when $\BR{\nn-1}$ is embedded in~$\BR\nn$ by adding an $\nn$th strand on the top of the diagrams. Going to the limit yields a left-invariant ordering on the limit group~$\BR\infty$. Note that, for $\nn = 2$, the group~$\BR\nn$ is the free group generated by~$\sig1$, so it is isomorphic to the additive group of integers and the associated D-ordering corresponds to the usual ordering of integers via $\pp \mapsto \sigg1\pp$. 

For $\nn \ge 3$, the D-ordering of $\nn$-strand braids is not right-invariant, and it is actually easy to show that no left-invariant ordering of~$B_\nn$ may be right-invariant. However, Laver proved

\begin{thrm}[Laver, \cite{Lve}]
\label{T:LeftMultSig}
For all~$\br$ in~$B_\nn$ and $\ii$ in~$\{1\wdots \nn-1\}$, the relation $\br \lD \sig\ii\br$ is satisfied.
\end{thrm}

In other words, whereas $\brr \lD \brr'$ does not imply $\brr\br \lD \brr'\br$ in general, $1 \lD \sig\ii$ does imply $1\br \lD \sig\ii\br$ for every braid~$\br$. Laver's proof of Theorem~\ref{T:LeftMultSig} relies on colouring braids (in the sense of Subsection~\ref{SS:Colouring}) using elements of free LD-systems and developing a fine combinatorial analysis of the latter structures by means of normal forms of their elements introduced by tricky recursive definitions---a quite delicate argument actually.

Let us say that a word~$\ww$ is a \emph{subword} of another word~$\ww'$ if $\ww'$ can be obtained from~$\ww$ by inserting letters, not necessarily in adjacent positions. Then Theorem~\ref{T:LeftMultSig} directly implies that the D-ordering has what is usually called the Subword Property: 

\begin{coro}[Laver, \cite{Lve}]
\label{C:Subword}
If $\br, \br'$ are braids and some braid word representing~$\br$ is a subword of some braid word representing~$\br$, then $\br \leD \br'$ holds.
\end{coro}

\begin{proof}
For an induction, it is sufficient to show that the conjunction of $\br = \br_1 \br_2$ and $\br' = \br_1 \sig\ii \br_2$ implies $\br \lD \br'$. Now Theorem~\ref{T:LeftMultSig} implies $\br_2 \lD \sig\ii \br_2$, whence $\br_1\br_2 \lD \br_1\sig\ii \br_2$ since $\lD$ is invariant under left-multiplication.
\end{proof}

The Subword Property directly implies that every conjugate~$\br'$ of a positive braid, that is, every braid of the form $\brr\inv\br\brr$ with $\br \in \BP\nn$, satisfies $\br' \gD 1$, since we can write $\brr\inv\br\brr \gD \brr\inv\brr = 1$. It follows in turn that $\br \gD 1$ is true for every \emph{quasipositive} braid~$\br$, the latter being defined as a braid that can be expressed as a product of conjugates of positive braids~\cite{Ore}.

Using the Subword Property in a more tricky way, one shows the following property that involves a sort of shifted conjugacy. 

\begin{coro}\cite[Lemma~3.5]{Dgb}
Let $\sh$ be the shift endomorphism of~$\BR\infty$ that maps~$\sig\ii$ to~$\sig{\ii+1}$ for every~$\ii$. Then, for every braid~$\br$, one has $\sh(\br) \, \sig1 \gD \br$.
\end{coro}

However, the most promising consequence of Laver's result is that some fragments of the D-ordering are well-orderings, that is, every nonempty subset must have a smallest element.

\begin{coro}[Laver, \cite{Lve}]
\label{C:WellOrder}
For every~$\nn$, the restriction of the D-ordering to~$\BP\nn$ is a well-ordering.
\end{coro}

\begin{proof}
By a celebrated result of Higman \cite{Hig}, an infinite set of words over a finite alphabet necessarily contains two elements~$\ww, \ww'$ such that $\ww$ is a subword of~$\ww'$. Let $\br_1, \br_2, ...$ be an infinite sequence of braids in~$\BP\nn$. Our aim is to prove that this sequence is not strictly decreasing. For each~$\pp$, choose a positive braid word~$w_\pp$ representing~$\br_\pp$. There are only finitely many $n$-strand braid words of a given length, so, for each~$\pp$, there exists~$\pp' > \pp$ such that $w_{\pp'}$ is at least as long as~$w_\pp$. So, inductively, we can extract a subsequence $w_{\pp_1}, w_{\pp_2}, ...$ in which the lengths are non-decreasing. If the set $\{w_{\pp_1}, w_{\pp_2}, ...\}$ is finite, there exist $\kk, \kk'$ such that $w_{\pp_\kk}$ and~$w_{\pp_{\kk'}}$ are equal, and then we have $\br_{\pp_\kk} = \br_{\pp_{\kk'}}$. Otherwise, by Higman's result, there exist $\kk, \kk'$ such that $w_{\pp_\kk}$ is a subword of~$w_{\pp_{\kk'}}$, and, by construction, we must have $\pp_\kk \le \pp_{\kk'}$. By Corollary~\ref{C:Subword}, this implies $\br_{\pp_\kk} \leD \br_{\pp_{\kk'}}$ in~$\BP\nn$. So, in any case, the sequence $\br_1, \br_2, ...$ is not strictly decreasing.
\end{proof}

The well-order property established by Laver for~$\BP\nn$ is a strong statement. As a general matter of fact, the D-ordering on~$\BR\nn$ is an intricate relation for $\nn \ge 3$: it is not Archimedean (there exist $\br, \br'$ such that $\br^\pp \lD \br'$ holds for every~$\pp$), it is not Conradian (there exist $\br, \br'$ such that $\br'\br^\pp \lD \br$ holds for every~$\pp$), it has infinite ascending and descending sequences, \textit{etc}. By contrast, Laver's result shows that forgetting about non-positive braids yields a very simple ordering, in particular one where the position of an element can be specified using just an ordinal, see Figure~\ref{F:WellOrder}.

\begin{figure}[htb]
\begin{picture}(125,20)(0,-2)
\psline{->}(5,0)(125,0)
\put(113,2.5){$(\BP3, \lD)$}
\pscircle[linewidth=.8pt, fillstyle=solid](5,0){2pt}
\put(4,2.5){$1$}
\pscircle[linewidth=.8pt, fillstyle=solid](15,0){2pt}
\put(14,2.5){$\sig2$}
\pscircle[linewidth=.8pt, fillstyle=solid](25,0){2pt}
\put(23.5,2.5){$\sigg22$\ ...}
\pscircle[linewidth=.8pt, fillstyle=solid](85,0){2pt}
\put(84,2.5){$\sig1$}
\pscircle[linewidth=.8pt, fillstyle=solid](95,0){2pt}
\put(92,2.5){$\sig1\!\sig2$}
\pscircle[linewidth=.8pt, fillstyle=solid](105,0){2pt}
\put(102,2.5){$\sig1\!\sigg22$\ ...}
\psline[linecolor=white, fillstyle=solid,fillcolor=lightgray](34,-2)(83,-2)(83,2)(34,2)(34,-2)

\put(0,10){$...$}
\psline{->}(3,10)(125,10)
\put(113,12.5){$(\BR3, \lD)$}
\pscircle[linewidth=.8pt, fillstyle=solid](5,10){2pt}
\put(4,12.5){$1$}
\pscircle[linewidth=.8pt, fillstyle=solid](15,10){2pt}
\put(14,12.5){$\sig2$}
\pscircle[linewidth=.8pt, fillstyle=solid](25,10){2pt}
\put(23.5,12.5){$\sigg22$\ ...}
\pscircle[linewidth=.8pt, fillstyle=solid](40,10){2pt}
\put(35,12.5){$\siginv2\!\sig1\!\siginv2$}
\pscircle[linewidth=.8pt, fillstyle=solid](50,10){2pt}
\put(49,13){$\downarrow$}
\put(46,17){$\siginv2\!\sig1$}
\pscircle[linewidth=.8pt, fillstyle=solid](60,10){2pt}
\put(54,12.5){$\siginv2\!\sig1\!\sig2\ ...$}
\pscircle[linewidth=.8pt, fillstyle=solid](75,10){2pt}
\put(72,12.5){$\sig1\!\siginv2$}
\pscircle[linewidth=.8pt, fillstyle=solid](85,10){2pt}
\put(84,12.5){$\sig1$}
\pscircle[linewidth=.8pt, fillstyle=solid](95,10){2pt}
\put(92,12.5){$\sig1\!\sig2$}
\pscircle[linewidth=.8pt, fillstyle=solid](105,10){2pt}
\put(102,12.5){$\sig1\!\sigg22$\ ...}
\end{picture}
\caption{\sf\smaller Restricting to positive braids changes the ordering: for instance, in $(\BP3, \lD)$, the braid~$\sig1$ is the limit of~$\sigg2\pp$, whereas, in~$(\BR3, \lD\nobreak)$, it is an isolated point with immediate predecessor~$\sig1\siginv2$; the grey part in~$\BR3$ includes infinitely many braids, such as $\siginv2\sig1$ and its neighbours---and much more---but none of them lies in~$\BP3$.}
\label{F:WellOrder}
\end{figure}
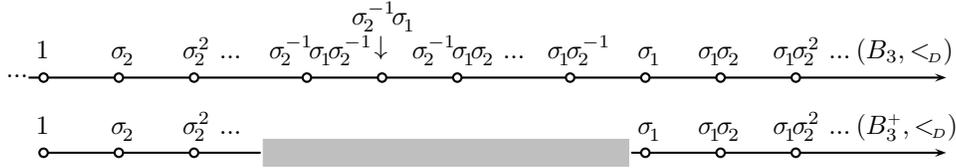

Among the standard consequences of the well-order property, we deduce

\begin{coro}
Every nonempty subset of~$\BP\nn$ is either cofinal or it has a least upper bound inside~$(\BP\nn, \lD)$.
\end{coro}

Indeed, for $\XX $ included in~$\BP\nn$, unless $\XX$ is unbounded in~$\BP\nn$, the set of all upper bounds of~$\XX$ is nonempty, hence it admits a least element.

Before turning to further results, let us conclude this subsection with a conjecture of R.\,Laver that involves braids and extends his well-order result. We saw in Lemma~\ref{L:ActionBis} and Lemma~\ref{L:PartialAction} that, whenever $(\SS, \op)$ is a left-cancellative LD-system, one can define a partial Hurwitz action of~$\BR\nn$ on~$\SS^\nn$. For every sequence~$\vec\aa$ in~$\SS^\nn$, we can then consider the family
$$\DD_\SS(\vec\aa) = \{\br \in \BR\nn \mid \vec\aa \act \br \text{\ is defined\,}\}.$$ 
As the action of positive braids is always defined, we always have $\BP\nn \subseteq \DD_\SS(\vec\aa)$. In some cases~\cite{Lar}, the family~$\DD_\SS(\vec\aa)$ reduces to~$\BP\nn$ and, therefore, Corollary~\ref{C:WellOrder} says that the restriction of the D-ordering to this family $\DD_{\BR\infty}(1 \wdots 1)$ is a well-order.

\begin{conj}[Laver, private communication]
If $(\SS, \op)$ is a free LD-system, the restriction of the D-ordering to every family of the form $\DD_\SS(\vec\aa)$ is a well-order.
\end{conj}

The conjecture remains open when $\vec\aa$ has length~3 and more. As noted by R.\,Laver, the above braid formulation is equivalent to a formulation involving free LD-systems only, and connected with the results of~\cite{LaMi1} and~\cite{LaMi2}. Let us also mention a similar conjecture where free LD-systems are replaced with the (left-cancellative) LD-system~$(\BR\infty, \op)$ where $\op$ is the shifted conjugacy operation
\begin{equation}
\label{E:BraidLDOp}
\br \op \brr = \br \cdot \sh(\brr) \cdot \sig1 \cdot \sh(\br)\inv,
\end{equation}
with $\sh$ the endomorphism that maps~$\sig\ii$ to~$\sig{\ii+1}$ for every~$\ii$. This conjecture is also open so far.

\subsection{Further refinements}
\label{SS:Refinements}

We now report about some subsequent results, mainly by S.\,Burckel, J.\,Fromentin, and the author, that made the description of the braid well-order more precise than the original abstract argument of R.\,Laver. 

The restriction of the D-ordering to the monoid~$\BP\infty$ of positive braids on an unbounded number of strands is not a well-ordering since it contains the descending sequence $\sig1 \gD \sig2 \gD \pdots$\,. However, it is easy, and technically convenient, to reverse the role of left and right in braid diagrams and to obtain a well-ordering on~$\BP\infty$.

\begin{defi}
For $\nn \ge 2$, let $\Flip\nn$ be the automorphism of the group~$\BR\nn$ (`\emph{flip} automorphism') that maps~$\sig\ii$ to $\sig{\nn-\ii}$ for every~$\ii$. For~$\br, \br'$ in~$\BR\nn$, we write $\br \lDf \br'$ for $\Flip\nn(\br) \lD \Flip\nn(\br')$. The relation~$\lDf$ is called the \emph{flipped D-ordering} on~$\BR\nn$.
\end{defi}

It is straightforward to check that the relation~$\lDf$ is a left-invariant linear ordering on~$\BR\nn$, and that it is independent of~$\nn$ in that, for $\br, \br'$ in~$\BR\nn$, the relation~$\br \lDf \br'$ holds in~$\BR\nn$ if and only if it holds in~$\BR{\nn'}$ for any~$\nn' \ge \nn$. When compared with the D-ordering, the flipped D-ordering amounts to exchanging left and right: $\br \lDf \br'$ holds if and only if the quotient-braid $\br\inv \br'$ admits an expression in which the generator~$\sig\ii$ with \emph{largest} index occurs positively only. In particular, we have $\sig1 \lDf \sig2 \lDf \pdots$. The benefit of considering~$\lDf$ instead of~$\lD$ is to give an improved picture of the way the monoids~$\BP\nn$ embed into one another, see Figure~\ref{F:Flip}. Indeed, one shows:

\begin{prop}\cite[Proposition~II.2.10]{Dhr}
\label{P:Segment}
The restriction of the flipped D-ordering of~$\BR\infty$ to~$\BP\infty$ is a well-ordering and, for every~$\nn$, the set~$\BP\nn$ is the initial segment of $(\BP\infty, \lDf)$ determined by~$\sig\nn$, that is, we have $\BP\nn = \{\br \in \BP\infty \mid \br \lDf \sig\nn\}$.
\end{prop}

\begin{figure}[htb]
\begin{picture}(120,25)
\psline{->}(0,20)(120,20)
\pscircle[linewidth=.8pt, fillstyle=solid](0,20){2pt}
\put(-1,22.5){$1$}
\pscircle[linewidth=.8pt, fillstyle=solid](5,20){2pt}
\put(4,22.5){$\sig1$}
\pscircle[linewidth=.8pt, fillstyle=solid](30,20){2pt}
\put(29,22.5){$\sig2$}
\pscircle[linewidth=.8pt, fillstyle=solid](60,20){2pt}
\put(59,22.5){$\sig3$}
\pscircle[linewidth=.8pt, fillstyle=solid](90,20){2pt}
\put(89,22.5){$\sig4$\quad...}
\put(106,22.5){$(\BP\infty, \lDf)$}
\put(-1,18){$\underbrace{\hbox to 29mm{\hfill}}$}
\put(-3,13){positive $2$-strand braids}
\put(-1,12){$\underbrace{\hbox to 59mm{\hfill}}$}
\put(12,7){positive $3$-strand braids}
\put(-1,6){$\underbrace{\hbox to 89mm{\hfill}}$}
\put(27,1){positive $4$-strand braids, etc.}
\end{picture}
\caption{\sf\smaller The well-ordered set~$(\BP\infty, \lDf)$: an increasing union of end-extensions, in which $\BP\nn$ is the initial segment determined by~$\sig\nn$.}
\label{F:Flip}
\end{figure}
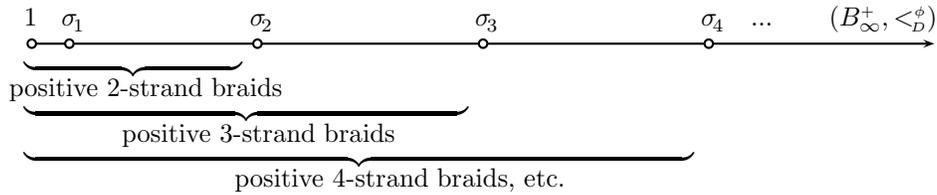

Laver's proof that the restriction of the D-ordering to positive braids is a well-ordering is indirect, and it remains ineffective in that it does not specifies the order type of $(\BP\nn, \lD)$ or $(\BP\infty, \lDf)$. These natural questions have been solved.

\begin{prop}[Burckel, \cite{Bur}]
\label{P:Burckel}
For every~$\nn$, the order type of $(\BP\nn, \lD)$ is $\omega^{\omega^{\nn-2}}$.
\end{prop}

Proposition~\ref{P:Burckel} also implies that the order type of $(\BP\nn, \lDf)$ is $\omega^{\omega^{\nn-2}}$ and, therefore, that of $(\BP\infty, \lD)$ is $\omega^{\omega^\omega}$, the upper bound of $\omega^{\omega^{\nn-2}}$ when $\nn$ goes to infinity.

Burckel's result relies on an intricate inductive argument, which assigns to every $\nn$-strand braid word a finite rooted tree with uniform height~$\nn - 2$, so that, for $\br, \br'$ in~$\BP\nn$, the relation $\br \lD \br'$ holds if and only if the $\SL$-minimal tree representing~$\br$ is $\SL$-smaller than the $\SL$-minimal tree representing~$\br'$, where a height~$\mm$ rooted tree is considered to be a finite sequence of height~$\mm-1$ rooted trees, and a height~$\mm$ rooted tree $(\TT_1 \wdots \TT_\ell)$ is declared $\SL$-smaller than another height~$\mm$ rooted tree~$\TT' = (\TT'_1 \wdots \TT'_{\ell'})$, if $\ell < \ell'$ holds, or if $\ell = \ell'$ holds and there exists~$\ii$ such that $\TT_\jj = \TT'_\jj$ holds for~$\jj < \ii$ and $\TT_\ii$ is $\SL$-smaller than~$\TT'_\ii$. 

In Burckel's approach, the $\SL$-minimal tree representing a braid~$\br$ appears as the terminal point of a recursive reduction process and it is not easily determined. The situation was made simpler when the simple connection between~$(\BP\nn, \lD)$ and~$(\BP{\nn-1}, \lD)$ stated in Proposition~\ref{P:Splitting} below was found, leading to considering the ordering of~$\BP\nn$ as an iterated extension of the ordering of~$\BP2$, that is, of the standard ordering of natural numbers.

If $\br, \br'$ are positive braids, one says that $\br$ \emph{right-divides}~$\br'$ if there exists a positive braid~$\brr$ satisfying $\br' = \brr \br$. Garside's theory of braids \cite{Gar} implies that every braid in~$\BR\nn$ admits a unique maximal right-divisor lying in~$\BR{\nn-1}$, namely the least common left-multiple of all right-divisors of~$\br$ lying in~$\BR{\nn-1}$. Iterating the result, one obtains a decomposition of every positive $\nn$-strand braid in terms of a sequence of positive $(\nn-1)$-strand braids. The result is then that, in terms of such decompositions, the flipped D-ordering on~$\BP\nn$ is the $\SL$-extension of the flipped D-ordering on~$\BP{\nn-1}$.

\begin{prop}\cite{Dho}
\label{P:Splitting}
For $\nn \ge 3$ and $\br$ in~$\BP\nn$, define the $\Flip\nn$-splitting of~$\br$ to the (unique) sequence $(\br_\pp \wdots \br_1)$ in~$\BP{\nn-1}$ such that, for each~$\rr$, the braid~$\br_\rr$ is the maximal right-divisor of $\brr_{\rr-1}$ that lies in~$\BP{\nn-1}$, where $\brr_\rr$ is inductively defined by $\br = \brr_\rr \br_\rr$ starting from $\br_0 = 1$. Then, for $\br, \br'$ in~$\BP\nn$ with $\Flip\nn$-splittings $(\br_\pp \wdots \br_1)$ and $(\br'_{\pp'} \wdots \br'_1)$, the relation $\br \lDf \br'$ holds if and only if $(\br_\pp \wdots \br_1)$ is smaller than $(\br'_{\pp'} \wdots \br'_1)$ for the $\SL$-extension of~$(\BP{\nn-1}, \lDf)$.
\end{prop}

Saying that $(\br_\pp \wdots \br_1)$ is the $\Flip\nn$-splitting of a braid~$\br$ means that one has
\begin{equation}
\label{E:BOSplitting}
\br = \Flip\nn^{\pp-1}(\br_\pp) \cdot ... \cdot \Flip\nn(\br_2) \cdot \br_1,
\end{equation}
and $\sig1$ is the only generator~$\sig\ii$ that right-divides $\Flip\nn^{\nn-\rr}(\br_\pp) \pdots \Flip\nn(\br_{\rr+1}) \, \br_\rr$ for each~$\rr$. By iterating the decomposition process, one eventually obtains for every positive braid~$\br$ an expression in terms of shifted powers of~$\sig1$, that is, a distinguished expression by a braid word, called the \emph{alternating normal form} of~$\br$. 

As the right-divisibility relation of braids can be tested in linear time, the $\Flip\nn$-splitting of a positive braid can be computed in quadratic time and Proposition~\ref{P:Splitting} implies that, for every~$\nn$, the orderings~$\lDf$ and $\lD$ of~$\BR\nn$ can be recognized in quadratic time.

One of the nice consequences of Laver's well-ordering result is that every positive braid can be characterized by a unique parameter, namely the ordinal that describes its rank in the well-order~$(\BP\infty, \lDf)$: for instance, the rank of the trivial braid~$1$ is~$0$, that of~$\sig1$, the immediate successor of~$1$, is~$1$ and, for $\ii \ge 2$, the rank of~$\sig\ii$ is the length of the initial segment determined by~$\sig\ii$, which is~$\BP\ii$ by Proposition~\ref{P:Segment}, hence this rank is $\omega^{\omega^{\ii-1}}$ by Proposition~\ref{P:Burckel}. It is then natural to ask for a complete explicit description of the rank function. The latter is \emph{not} an algebraic homomorphism with respect to the ordinal sum: in general, the rank of~$\br_1\br_2$ is not the sum of the ranks of~$\br_1$ and~$\br_2$. This happens to be true when $\br_2$ is $\sig1$, which has rank~$1$ but, for instance, the rank of~$\sig1\sig2$ turns out to be~$\omega^2$, which is not $1 + \omega$ (that is, $\omega$) although the rank of~$\sig2$ is~$\omega$. The problem essentially amounts to recognizing which braid words are alternating normal; in the case of $3$-strand braids, the answer is simple:

\begin{prop}\cite[Proposition 6.7]{Dho} 
\label{P:Rank3}
Put $\eps_1 = 0$, $\eps_2 = 1$, and $\eps_\rr = 2$ for $\rr \ge 3$. Then every braid in~$\BP3$ admits a unique expression $\sigg{\smash{\mathrm{parity}(\pp)}}{\ee_\pp} \pdots \sigg2{\ee_2} \sigg1{\ee_1}$ with $\ee_\pp \ge 1$, and $\ee_\rr \ge \eps_\rr$ for~$\rr < \pp$; its rank in $(\BP3, \lDf)$ is then
\begin{equation}
\label{E:Rank3}
\omega^{\pp-1} \cdot \ee_\pp + \sum_{\pp > \rr\ge1} \omega^{\rr-1} \cdot (\ee_\rr - \eps_\rr).
\end{equation}
\end{prop}

For instance, the alternating normal form of Garside's fundamental braid~$\Delta_3$ is $\sig1 \sig2 \sig1$, as the latter word satisfies the defining inequalities of Proposition~\ref{P:Rank3}, contrary to $\sig2 \sig1 \sig2$, that is, $\sigg21 \sigg11 \sigg21 \sigg10$, in which the third exponent from the right, namely~$1$, is smaller than the minimal legal value~$\eps_3 = 2$. So, in this case, the sequence $(\ee_\pp \wdots \ee_1)$ is $(1, 1, 1)$, and, applying~\eqref{E:Rank3}, we deduce that the rank of~$\Delta_3$ in $(\BP3, \lDf)$---hence in $(\BP\infty, \lDf)$ as well---is $\omega^2+1$.

In the general case, only partial results are known: for instance, it is shown in~\cite{Dho} that the family of all alternating normal $\nn$-strand braid words is recognized by a finite state automaton and, in~\cite{BurR}, S.\,Burckel describes a recursive procedure for determining the rank in~$\BP\nn$.

We conclude with extensions of the previous results involving other submonoids of the braid groups. It turns out that the argument used to establish that the restriction of the D-ordering to the monoid~$\BP\nn$ is a well-ordering works for other submonoids:

\begin{coro}
\label{C:WellOrderBis}
Assume that $\MM$ is a submonoid of~$\BR\infty$ that is generated by finitely many elements, each of which is a conjugate of some generator~$\sig\ii$. Then the restriction of the D-ordering  to~$\MM$ is a well-order.
\end{coro}

\begin{proof}
The argument is the same as for Corollary~\ref{C:WellOrder}: assuming that $\MM$ is generated by~$\br_1 \wdots \br_\pp$, it suffices to show that the Subword Property is valid for the words in the alphabet $\{\br_1 \wdots \br_\pp\}$ and, for this, it is enough to show that $\br \lD \br_\kk \br$ holds for every positive braid~$\br$. Now, assuming $\br_\kk = \brr\inv \sig\ii \brr$, Theorem~\ref{T:LeftMultSig} gives $\brr \br \lD \sig\ii\brr \br$, whence $\br \lD \brr\inv\sig\ii\brr \br$, for every~$\br$.
\end{proof}

The hypothesis that the monoid~$\MM$ is finitely generated is crucial in Corollary~\ref{C:WellOrderBis}. For instance, the descending sequence $\sig1 \gD \sig2 \gD... $ witnesses that the submonoid~$\BP\infty$ of~$\BR\infty$ is not well-ordered by the D-ordering. Such phenomena already occur inside~$\BR3$: for instance, the submonoid of~$\BR3$ generated by all conjugates~$\sigg2{-\pp} \sig1 \sigg2\pp$ of~$\sig1$---and, more generally, the submonoid of all quasipositive $\nn$-strand braids---contains the infinite descending sequence $\sig1 \gD \siginv2\sig1\sig2 \gD \sigg2{-2} \sig1 \sigg22 \gD \pdots$\,.

A typical example of a monoid eligible for Corollary~\ref{C:WellOrderBis} is the \emph{dual braid monoid}~$\BKL\nn$, which is the submonoid of~$\BR\nn$ generated by the $\nn\choose2$ braids of the form $\sig\ii \pdots \sig\jj \siginv{\jj-1} \pdots \siginv\ii$, the `band' or `Birman--Ko--Lee' generators~\cite{BKL}. J.\,Fromen\-tin showed in~\cite{FroWell} that the order type of $(\BKL\nn, \lDf)$ is $\omega^{\omega^{\nn-2}}$, using a characterization of the restriction of the (flipped) D-ordering to~$\BKL\nn$ in terms of a normal from (`rotating normal form') that is analogous to the alternating normal form of Proposition~\ref{P:Splitting} but involves an order~$\nn$ automorphism analogous to a rotation instead of the order~$2$ automorphism~$\Flip\nn$ that is analogous to a symmetry~\cite{FroShort}. At the technical level, the properties of the rotating normal form are often nicer than those of the alternating normal form.

Another indirect outcome of Laver's result is the investigation of the well-founded\-ness of alternative braid orderings. There exists an uncountable family of left-invariant linear orderings on the braid group~$\BR\nn$~\cite[Chapter~XIV]{Dhr}. Most of them do not induce well-orderings on the braid monoid~$\BR\nn$, but at least all the orderings stemming from the hyperbolic geometry approach suggested by W.\,Thurston and investigated in~\cite{ShW} do, and it was recently shown that the associated order type is again $\omega^{\omega^{\nn-2}}$~\cite{ItoThurston, ItoDual}. 

\subsection{Applications to unprovability statements}
\label{SS:Unprovability}

The order type of the well-ordering on~$\BP\nn$, namely $\omega^{\omega^{\nn-2}}$, is a (relatively) large ordinal: although not extremely large in the hierarchy of countable ordinals, it is large enough to give rise to nontrivial unprovability statements. The principle is that, although the well-order property forbids that infinite descending sequences exist, there exist nevertheless finite descending sequences that are so long that their existence cannot be proved in weak logical systems. 

It is well-known that there exist strong limitations about the sentences possibly provable in a given formal system, starting with G\"odel's famous theorems implying that certain arithmetic sentences cannot be proved in the first-order Peano system. However, the G\"odel sentences have a strong logical flavour and they remain quite remote from the sentences usually considered by mainstream mathematicians. It is therefore natural to look for further sentences that are unprovable in the Peano system, or in other formal systems, and, at the same time, involve objects and properties that are both simple and natural. Typical results in this direction involve finite combinatorics, well-quasiorders, and the Ramsey Theory~\cite{Bov, Fri, Wei2}.

We shall mention some results along this line of research that involve the D-ordering of braids. Here we shall restrict to the case of $3$-strand braids and refer to~\cite{Dhq} for details and extensions. In order to construct a long sequence of braids, we start with an arbitrary braid in~$\BP3$ and then repeat some transformation until, if ever, the trivial braid is obtained. Here, the transformation at step~$\tt$ will consist in removing one crossing, but, in all cases but one, introducing $\tt$~new crossings. It is reminiscent of Kirby--Paris' Hydra Game~\cite{KiP}, with Hercules chopping off one head of the Hydra and the Hydra sprouting $\tt$~new heads. The paradoxical result is that, contrary to what examples suggest, one always reaches the trivial braid after finitely many steps.

\begin{defi}
\label{D:APGame}
For $\br$ is a nontrivial positive $3$-strand braid, and $\tt$ a positive integer, define $\br\{\tt\}$ to be the braid represented by the following diagram: in the alternating normal diagram of~$\br$, we remove one crossing in the critical block, defined to be the rightmost block whose size is not the minimal legal one, and add $\tt$~crossings in the next block, if it exists, that is, if the critical block is not the final block of~$\sig1$. The \emph{$\GGG_3$-sequence from~$\br$} is defined by $\br_0 = \br$ and $\br_\tt = \br_{\tt-1}\{\tt\}$ for $\tt \ge 1$; it stops when the trivial braid~$1$ is possibly obtained.
\end{defi}

It is easy to check that the $\GGG_3$-sequence from~$\sigg22\sigg12$ has length~14: it consists of $\sigg22\sigg12$, $\sigg22\sig1$, $\sigg22$, $\sig2\sigg13$, $\sig2\sigg12$, $\sig2\sig1$, $\sig2$, $\sigg17$, $\sigg16$, $\sigg15$, $\sigg14$, $\sigg13$, $\sigg12$, $\sig1$, and finally~$1$. Similarly, the $\GGG_3$-sequence from~$\Delta_3$ has length~$30$. Not all examples are so easy: starting from $\sigg12\sigg22\sigg12$, a braid with six crossings only, one does reach the trivial braid, but after $90,159,953,477,630$~steps.

\begin{prop}[Carlucci, D., Weiermann \cite{Dhq}]
\label{P:Unprovability}
\ITEM1 For every braid~$\br$ in~$\BP3$, the $\GGG_3$-sequence from~$\br$ is finite, that is, there exists a finite number~$\tt$ for which $\br_\tt = 1$ holds.

\ITEM2 The statement of~\ITEM1 is an arithmetic statement that cannot be proved from the axioms of the system~$\IS1$.
\end{prop}

Although braids are not natural numbers, one can encode braids and their basic operations using natural numbers and the usual arithmetic operations. Therefore, it makes sense to speak of braid properties that can be proved from a certain system of arithmetical axioms: by this we mean that some reasonable encoding of braids by natural numbers has been fixed once for all and we consider the arithmetic counterpart of the braid property we have in mind.

The standard first-order Peano axiomatization of arithmetic~$\PA$ consists of basic axioms involving addition and multiplication, plus the induction scheme, which asserts that, for each first-order formula~$\Phi(\xx)$ involving~$+, \times$ and~$<$, the conjunction of $\Phi(0)$ and $\forall\nn(\Phi(\nn) \Rightarrow \Phi(\nn{+}1))$ implies $\forall\nn(\Phi(\nn))$. Then $\IS\kk$ is the subsystem of~$\PA$ in which the induction principle is restricted to formulas of the form $\exists\xx_1 \forall\xx_2 \exists\xx_3 \pdots Q\xx_\kk(\Psi)$, where $Q$ is $\exists$ or $\forall$ according to the parity of~$\kk$ and $\Psi$ is a formula that only contains bounded quantifications $\forall\xx{<}\yy$ and $\exists\xx{<}\yy$.

\begin{proof}[Proof of Proposition~\ref{P:Unprovability} (sketch)]
For~\ITEM1, one shows that every $\GGG_3$-sequence is descending with respect to~$\lDf$, and the result then follows from Laver's well-order result (Corollary~\ref{C:WellOrder}). For~\ITEM2, in order to prove that a certain sentence~$\Phi$ is not provable from the axioms of~$\IS1$, it is sufficient to establish that, from~$\Phi$, and using arguments that can be formalized in~$\IS1$, one can prove the existence of a function that grows as fast as the Ackermann function. Now, if $T(\br)$ denotes the length of the $\GGG_3$-sequence from~$\br$, then the function $\pp \mapsto T(\Delta_3^\pp)$ actually grows as fast as the Ackermann function.
\end{proof}

The $\IS1$-unprovability result of Proposition~\ref{P:Unprovability} is directly connected with the order type~$\omega^\omega$ of the well-ordering on~$\BP3$. Similarly, the  order type $\omega^{\omega^\omega}$ of the well-ordering on~$\BP\infty$ induces a connection with the stronger system~$\IS2$: in~\cite{Dhq} a certain notion of $\GGG_\infty$-sequence in~$\BP\infty$ is defined so that, as can be expected, the analog of Proposition~\ref{P:Unprovability} is established, namely every $\GGG_\infty$-sequence is finite but that result cannot be proved from~$\IS2$. As the order-type of~$(\BP\infty, \lDf)$ is larger than that of~$(\BP3, \lDf)$, the $\GGG_\infty$-sequences can be made longer than the $\GGG_3$-sequences, so proving their finiteness is more difficult and requires a stronger logical context.

Further results involve the transition between provability and unprovability, which turns out to happen at a level that can be described precisely. To this end, one considers the length of descending sequences of braids that admit some bounded Garside complexity. Let $\Delta_3$ be the positive $3$-strand braid $\sig1\sig2\sig1$. Garside \cite{Gar} showed that every braid in~$\BP3$ right-divides some power~$\Delta_3^\dd$. Define the \emph{degree}~$\deg\br$ of a braid~$\br$ to be the least integer~$\dd$ such that $\br$ right-divides~$\Delta_3^\dd$. Then, for~$\ff$ a fixed function on the integers, we consider (the length of) the descending sequences $(\br_0 \wdots \br_\NN)$ in~$(\BP3, \lDf)$ satisfying $\deg{\br_\tt} \le \dd + \ff(\tt)$ for every~$\tt$, that is, the descending sequences whose complexity is, in a sense, bounded by~$\ff$. If $\ff$ is constant, the number of braids~$\br$ satisfying $\deg\br \le \dd + \ff(\tt)$ is finite, so the length of a sequence as above is certainly bounded. One can show using K\"onig's Lemma that, for every function~$\ff$, the length of a sequence as above is bounded by some constant (depending on~$\dd$). The question is whether this can be proved in the system~$\IS1$, and the result is that there exists a quick transition phase between $\IS1$-provability and $\IS1$-unprovability. Indeed, using~$\Ack_\rr$ for the functions defined by the double recursion rules: $\Ack_0(\xx) = \xx+1$, $\Ack_\rr(0) = \Ack_{\rr-1}(1)$, and $\Ack_\rr(\xx+1) = \Ack_{\rr-1}(\Ack_\rr(\xx))$ for $\rr \ge 1$ and $\Ack$ for the diagonal function defined by $\Ack(\xx) = \Ack_\xx(\xx)$ (`Ackermann function'), and using $\ff\inv$ for the functional inverse of~$\ff$, we have

\begin{prop}\cite{Dhq}
\label{P:APTransition}
Denote by~$\WO\ff$ the statement: 

{\leftskip5mm\rightskip5mm\noindent ``For every~$\dd$, there exists~$\NN$ such that every descending sequence $(\br_0, \br_1, ...)$ in~$(\BP3, \lDf)$ satisfying $\deg{\br_\tt} \le \dd + \ff(\tt)$ for every~$\tt$ has length at most~$\NN$.''\par} 

\noindent Put $\ff_\rr(\xx) = \lfloor\!\!\sqrt[\Ack_\rr\inv(\xx)]{\xx}\rfloor$ for $\rr \ge 0$, and $\ff(\xx) = \lfloor\!\!\sqrt[\Ack\inv(\xx)]{\xx}\rfloor$. Then, for every~$\rr$, the principle~$\WO{\ff_\rr}$ is provable in~$\IS1$, but $\WO{\ff}$ is not provable in~$\IS1$. 
\end{prop}

The functions involved in Proposition~\ref{P:APTransition} all are of the form $\xx\mapsto \!\!\sqrt[g(\xx)]\xx$ where $g$ is a very slowly increasing function. The proof is a---rather sophisticated---mixture of combinatorial methods and specific results about the number of $3$-strand braids satisfying some order and degree constraints.

It is likely that a similar result involving~$\BP\infty$ and $\IS2$ could be established, but this was not made in~\cite{Dhq}. 

\subsection{Braid conjugacy}
\label{SS:Mu}

We conclude with applications of the well-order property of a different nature, namely those where the order is used to provide distinguished elements. As we shall see, not much is known so far, but the approach leads at the least to testable conjectures.

As a preliminary remark, let us mention that connections are known between the position of a braid~$\br$ in the D-ordering, typically the unique interval $[\Delta_\nn^{2\kk}, \Delta_\nn^{2\kk+1})$ it belongs to (the parameter~$\kk$ is then called the \emph{D-floor} of~$\br$), and various topological parameters associated with the link that is the closure of~$\br$ \cite{Mal, MaN, ItoClosed, ItoGenus}, see~\cite{Dhy}, but we shall not give details here as these results do not involve the well-order property.

By very definition, the well-order property asserts that every nonempty subset of~$\BP\infty$ contains a $\lDf$-minimal element, and that every nonempty subset of~$\BP\nn$ contains a $\lD$-minimal element. In this way, one obtains a natural way to distinguish an element in a family of positive braids. As the D-ordering of braids appears as canonical in that many different approaches lead to the same ordering, one may expect that the elements so identified enjoy good properties. 

The Conjugacy Problem for the group~$\BR\nn$, namely the question of algorithmically recognizing whether two braids are conjugated, is one of the main algorithmic questions involving braids. The question was shown to be decidable by F.A.\,Garside~\cite{Gar} but, in spite of many efforts, the best methods known so far in the case of $5$~strands and more have an exponential complexity with respect of the length of the input braid words---which led to proposing braid groups and conjugacy as a cryptographic platform~\cite{KLC, Dgw}.

Because of the specific properties of Garside's fundamental braid~$\Delta_\nn$, every braid of~$\BR\nn$ can be expressed as $\Delta_\nn^{-\dd} \br$ with~$\br$ in~$\BP\nn$, and any two braids are conjugated if and only if they are positively conjugated, that is, conjugated via a positive braid. It follows that, in order to solve the Conjugacy Problem of~$\BR\nn$, it is enough to solve the Conjugacy Problem of the monoid~$\BP\nn$. Now, the well-order property implies

\begin{lemm}
For every braid~$\br$ in~$\BP\nn$, the intersection of the conjugacy class of~$\br$ with~$\BP\nn$ contains a
unique $\lD$-minimal element~$\mu_\nn(\br)$.
\end{lemm}

Being able to algorithmically compute the function~$\mu_\nn$ would provide an immediate solution for the Conjugacy Problem of~$\BP\nn$, since $\br$ is conjugated to~$\br'$ if and only if $\mu_\nn(\br)$ and $\mu_\nn(\br')$ are equal. So the question is to compute the function~$\mu_\nn$.

At the moment, the question remains open but, at least in the case of $3$-strand braids, the simple connection between the (flipped) D-ordering and the alternating normal form of Proposition~\ref{P:Rank3} makes it realistic to explicitly compute the function~$\mu_3$. To this end, the obvious approach is to investigate the analog of the cycling and decycling operations of~\cite{ElM} with the Garside normal form replaced by the alternating normal form (or the rotating normal form), that is, the operations corresponding to $(\br_\pp \wdots \br_1) \mapsto (\br_1, \br_\pp \wdots \br_2)$ and $(\br_\pp \wdots \br_1) \mapsto (\br_{\pp-1} \wdots \br_1, \br_\pp)$. 

\begin{table}[htb]
\hspace{-10mm}\begin{tabular}{c|*{12}c}
\VR(4,2)$\br$
&${1}$
&${\sig1}$
&${\sigg12}$
&$\pdots$
&$\sig2$
&${\sig2\sig1}$
&${\sig2\sigg12}$
&$\pdots$
&$\sigg22$
&${\sigg22\sig1}$
&${\sigg22\sigg12}$
&$\pdots$
\\\hline
$\mathrm{rk}(\br)$
&\SMALL$0$
&\SMALL$1$
&\SMALL$2$
&
&\SMALL$\omega$
&\SMALL$\omega{+}1$
&\SMALL$\omega{+}2$
&
&\SMALL$\omega{\cdot}2$
&\SMALL$\omega{\cdot}2{+}1$
&\SMALL$\omega{\cdot}2{+}2$
\\ \hline\hline
\VR(3.5,0)$\mu_3(\br)$
&$\circlearrowright$
&$\circlearrowright$
&$\circlearrowright$
&
&$\sig1$
&$\circlearrowright$
&$\circlearrowright$
&
&$\sigg12$
&${\sig2\sigg12}$
&$\circlearrowright$
\end{tabular}
\vspace{1mm}

\begin{tabular}{*9c}
\VR(4,2)$\pdots$
&$\sigg23$
&${\sigg23\sig1}$
&${\sigg23\sigg12}$
&$\pdots$
&$\sig1\sig2$
&$\sig1\sig2\sig1$
&$\sig1\sig2\sigg12$
&$\pdots$
\\\hline
&\SMALL$\omega{\cdot}3$
&\SMALL$\omega{\cdot}3{+}1$
&\SMALL$\omega{\cdot}3{+}2$
&
&\SMALL$\omega^2$
&\SMALL$\omega^2{+}1$
&\SMALL$\omega^2{+}2$
\\\hline\hline
\VR(3.5,0)
&$\sigg12$
&${\sig2\sigg13}$
&${\sigg22\sigg13}$
&
&$\sig2\sig1$
&$\sig2\sigg12$
&$\sig2\sigg13$
\end{tabular}
\vspace{1mm}

\begin{tabular}{*9c}
\VR(4,2)$\pdots$
&$\sig1\sigg22$
&$\sig1\sigg22\sig1$
&$\sig1\sigg22\sigg12$
&$\pdots$
&$\sigg12\sig2$
&$\sigg12\sig2\sig1$
&$\sigg12\sig2\sigg12$
&$\pdots$
\\\hline
&\SMALL$\omega^2{+}\omega$
&\SMALL$\omega^2{+}\omega{+}1$
&\SMALL$\omega^2{+}\omega{+}2$
&
&\SMALL$\omega^2{\cdot}2$
&\SMALL$\omega^2{\cdot}2{+}1$
&\SMALL$\omega^2{\cdot}2{+}2$
\\\hline\hline
\VR(3.5,0)
&$\sig2\sigg12$
&$\sigg22\sigg12$
&$\sigg22\sigg13$
&
&$\sig2\sigg12$
&$\sig2\sigg13$
&$\sig2\sigg14$
\end{tabular}
\vspace{1mm}

\begin{tabular}{*9c}
\VR(4,2)$\pdots$
&$\sigg12\sigg22$
&$\sigg12\sigg22\sig1$
&$\sigg12\sigg22\sigg12$
&$\pdots$
&$\sigg12\sigg23$
&$\sigg12\sigg23\sig1$
&$\pdots$
\\\hline
&\SMALL$\omega^2{\cdot}2{+}\omega$
&\SMALL$\omega^2{\cdot}2{+}\omega{+}1$
&\SMALL$\omega^2{\cdot}2{+}\omega{+}2$
&
&\SMALL$\omega^2{\cdot}2{+}\omega{\cdot}2$
&\SMALL$\omega^2{\cdot}2{+}\omega{\cdot}2{+}1$
\\\hline\hline
\VR(3.5,0)
&$\sigg22\sigg12$
&$\sigg22\sigg13$
&$\sigg22\sigg14$
&
&$\sigg22\sigg13$
&$\sigg23\sigg13$
\end{tabular}
\vspace{1mm}

\begin{tabular}{*9c}
\VR(4,2)$\pdots$
&\hfill$\sigg12\sigg24$\hfill\null
&\hfill$\sigg12\sigg24\sig1$\hfill\null
&\hfill$\sigg12\sigg24\sigg12$\hfill\null
&\hfill$\pdots$\hfill\null
&\hfill$\sigg13\sig2$\hfill\null
&\hfill$\sigg13\sig2\sig1$\hfill\null
&\hfill$\pdots$\hfill\null
\\\hline
&\SMALL$\omega^2{\cdot}2{+}\omega{\cdot}3$
&\SMALL$\omega^2{\cdot}2{+}\omega{\cdot}3{+}1$
&\SMALL$\omega^2{\cdot}2{+}\omega{\cdot}3{+}2$
&
&\SMALL$\omega^2{\cdot}3$
&\SMALL$\omega^2{\cdot}3{+}1$
\\\hline\hline
\VR(3.5,0)
&\hfill$\sigg22\sigg14$\hfill\null
&\hfill$\sigg23\sigg14$\hfill\null
&\hfill$\sigg24\sigg14$\hfill\null
&
&\hfill$\sig2\sigg13$\hfill\null
&\hfill$\sig2\sigg14$\hfill\null
\end{tabular}
\vspace{1mm}

\begin{tabular}{*9c}
\VR(4,2)\hfill$\pdots$\hfill\null
&\hfill$\sigg13\sigg22$\hfill\null
&\hfill$\sigg13\sigg22\sig1$\hfill\null
&\hfill$\pdots$\hfill\null
&\hfill$\sig2\sigg12\sig2$\hfill\null
&\hfill$\sig2\sigg12\sig2\sig1$\hfill\null
&\hfill$\sig2\sigg12\sig2\sigg12$\hfill\null
&\hfill$\pdots$\hfill\null
\\\hline
&\SMALL$\omega^2{\cdot}3{+}\omega$
&\SMALL$\omega^2{\cdot}3{+}\omega{+}1$
&
&\SMALL$\omega^3$
&\SMALL$\omega^3{+}1$
&\SMALL$\omega^3{+}2$
&
\\\hline\hline
\VR(3.5,0)
&\hfill$\sigg22\sigg13$\hfill\null
&\hfill$\sigg22\sigg14$\hfill\null
&
&\hfill$\sigg22\sigg12$\hfill\null
&\hfill$\sig2\sigg14$\hfill\null
&$\circlearrowright$
&
\end{tabular}\hspace{-2cm}
\vspace{2mm}
\caption{\sf\smaller A few values of the function~$\mu_3$ on~$\BP3$: here the braids are enumerated in $\lDf$-increasing order, specified using their alternating normal form, and accompanied with their ordinal rank in the well-order; the symbol $\circlearrowright$ indicates the fixed points of~$\mu_3$, that is, the braids that are minimal in their conjugacy class.}
\label{T:Mu}
\end{table}

Some values of the function~$\mu_3$ are listed in Table~\ref{T:Mu}; these values should suggest both that $\mu_3$ is nontrivial but also that it obeys simple rules. For instance, a typical rule suggested by the computer experiments is the following formula (an $\nn$-strand version is easy to guess):

\begin{conj}[D., Fromentin, Gebhardt, \cite{Dhy}]
For every $\br$ in~$\BP3$, one has
$$\mu_3(\br \Delta_3^2) = \sig1\sigg22 \sig1 \cdot \mu_3(\br) \cdot \sigg12.$$
\end{conj}

It is reasonable to expect that an investigation of cycling and decycling for the alternating normal form would lead to a solution of that specific conjecture and, more generally, lead to the practical computation of the function~$\mu_3$ on~$\BP3$, and subsequently of the function~$\mu_\nn$ on~$\BP\nn$. 

If this program can be fulfilled, it would then become foreseeable to investigate similar questions for analogous functions in which the conjugacy relation is replaced with the Markov equivalence relation, that is, the equivalence relation on~$\BR\infty$, or rather on $\bigcup_\nn \BR\nn \times \{\nn\}$, generated by conjugacy together with the Markov transformation $(\br, \nn) \sim (\br\sigg\nn{\pm1}, \nn+1)$. It is well-known \cite{Bir} that the closures of two braid diagrams represent the same link if and only if the braids are Markov-equivalent; moreover, at the expense of taking into account a power of the braid~$\Delta_\nn^2$, one can always reduce to the case of positive braids. So, should the above approach turn out to be possible, one would associate with every link~$L$ a unique distinguished braid, namely the $(\BP\infty, \lDf)$-smallest positive braid in the equivalence class of the braids that represent~$L$---or, equivalently, the unique ordinal that is the rank of this distinguished braid in the well-order. Of course, the problem here is not the existence of the smallest braid or the ordinal (which is guaranteed by Laver's result), but its practical computability.

\begin{rema}
In the case of Markov-equivalence, the braid index and the braid length are not preserved, so an equivalence class is in general infinite and Laver's result is essential to ensure the existence of a smallest representative. However, in the case of conjugacy, the braid index and the length (of positive braids) are preserved, so the considered equivalence classes are finite: in this case, the existence of a smallest element is guaranteed for every linear ordering of braids, and Laver's result is important for a motivation, but it is not needed to ensure the existence of the function~$\mu_\nn$. As we cannot expect to solve the Conjugacy Problem for free, this suggests that the investigation of~$\mu_\nn$ may still require significant technical efforts. 
\end{rema}


\end{document}